\documentclass[12pt]{article}

\usepackage[normalem]{ulem}

\usepackage{amsmath,amsthm}
\usepackage{amsfonts,amsmath}
\usepackage{amssymb}
\usepackage{fullpage}
\usepackage{epsf}
\usepackage{color}
\usepackage{graphicx}
\usepackage{pdflscape}
\usepackage{rotating}
\usepackage{comment}
\usepackage{enumitem}

\usepackage{mathrsfs}
\usepackage{pgf,tikz}
\usepackage{multirow}

\newtheorem{theorem}{Theorem}[section]
\newtheorem{lemma}[theorem]{Lemma}

\newtheorem{corollary}[theorem]{Corollary}
\newtheorem{conjecture}[theorem]{Conjecture}
\newtheorem{question}[theorem]{Question}
\theoremstyle{definition}
\newtheorem{definition}[theorem]{Definition}
\newtheorem{example}[theorem]{Example}

\newtheorem{remark}[theorem]{Remark}

\renewcommand{\geq}{\geqslant}
\renewcommand{\leq}{\leqslant}

\def \cP {{\cal P}}

\def \cS {{\cal S}}

\definecolor{darkgreen}{RGB}{34, 150, 65}

\title{
Weak colourings of Kirkman triple systems
}

\author{Andrea C.~Burgess\thanks{Department of Mathematics and Statistics, University of New Brunswick, Saint John, NB, E2L~4L5, Canada.
\texttt{andrea.burgess@unb.ca}}
\and
Nicholas J.~Cavenagh\thanks{Department of Mathematics, The University of Waikato, Private Bag 3105, Hamilton 3240, New Zealand.
\texttt{nicholas.cavenagh@waikato.ac.nz}}
\and
Peter Danziger\thanks{Department of Mathematics, Toronto Metropolitan University, Toronto, ON, M5B~2K3, Canada.
\texttt{danziger@torontomu.ca}}
\and
David A.~Pike\thanks{Department of Mathematics and Statistics, Memorial University of Newfoundland, St.~John's, NL, A1C~5S7, Canada.
\texttt{dapike@mun.ca}}
}

\begin{document}

\maketitle

\begin{abstract}
A $\delta$-colouring of the point set of a block design is said to be {\em weak} if no block is monochromatic.
The {\em chromatic number} $\chi(S)$ of a block design $S$
is the smallest integer $\delta$ such
that $S$ has a weak $\delta$-colouring.
It has previously been shown that any Steiner triple system
has chromatic number at least $3$ and that for each $v\equiv 1$ or $3\pmod{6}$ there exists a Steiner triple system on $v$ points that has chromatic number $3$. Moreover, for each integer $\delta \geq 3$ there exist infinitely many Steiner triple systems with chromatic number $\delta$.

We consider colourings of the subclass of Steiner triple systems which are resolvable.
A {\em Kirkman triple system} consists of a resolvable Steiner triple system together with a partition of its blocks into parallel classes.
We show that for each $v\equiv 3\pmod{6}$ there exists a Kirkman triple system on $v$ points with chromatic number $3$.  We also show that for each integer $\delta \geq 3$, there exist infinitely many Kirkman triple systems with chromatic number $\delta$.  We close with several open problems.
\end{abstract}



\section{Introduction}

A $(v,k,\lambda)$-BIBD ({\em balanced incomplete block design}) is a pair
$(V,{\mathcal B})$  such that
$V$
is a set of
$v$
elements (called {\em points}) and
${\mathcal B}$
is a collection of
$k$-element subsets of
$V$
(called {\em blocks}) such that each unordered pair
of points in
$V$
is contained in exactly
$\lambda$
blocks in
${\mathcal B}$.
When $k=3$, the blocks of ${\mathcal B}$ are often called {\em triples}.
A $(v,3,\lambda)$-BIBD is also known as a {\em Steiner triple system} or STS$(v)$ if $\lambda=1$.
It is well known that an STS$(v)$ exists if and only if $v\equiv 1$ or $3\pmod{6}$~\cite{Kirkman1847}; such integers $v$ are the {\em admissible} orders for Steiner triple systems.

Let $S=(V,{\mathcal B})$ be a $(v,k,\lambda)$-BIBD.
A {\em colouring} of $S$ is any mapping $\phi$ from
$V$ to a colour set $C$.
If $|C|=\delta$
we say that $\phi$ is a {\em $\delta$-colouring}.
Given any $\delta$-colouring $\phi:V \rightarrow C$ and point $x \in V$, if $\phi(x)=c$, then we say that $x$ has colour $c$.  The {\em colour class} of $c \in C$ is the set of points of colour $c$, i.e.\ $\phi^{-1}(c) = \{ x \in V \mid \phi(x)=c\}$.  If $\phi$ is surjective, then the colour classes partition the point set $V$.
The {\em colour type} associated with a colouring corresponds to a multiset of integers which give the sizes of the colour classes. We employ partition notation to denote the colour type; e.g. the type $\{3,3,3,4\}$ may be described as $3^34^1$.
If the elements of the colour type differ by at most one, we say that the colouring is {\em equitable}.

A block $B\in {\mathcal B}$ is said to be {\em monochromatic} if
each point of the block $B$ is mapped to the same colour.
We say that $\phi$ is a {\em weak} colouring of a design if there are no monochromatic blocks.
The {\em weak chromatic number} $\chi(S)$ is the smallest $\delta$ such that
the design $S$ has a weak $\delta$-colouring.  In this paper, all colourings considered are weak, so we will hereafter omit the word ``weak''; thus, in the remainder of the paper we refer to a weak $\delta$-colouring as simply a {\em $\delta$-colouring} and $\chi(S)$ as the {\em chromatic number} of $S$.  If $S$ has chromatic number $\delta$, then we say that $S$ is {\em $\delta$-chromatic}.

No Steiner triple system is $2$-chromatic, save for the degenerate case of an STS($3$)~\cite{Pelikan1970,Rosa1970}.
There exists a $3$-chromatic STS$(v)$ whenever $v\equiv 1$ or $3\pmod{6}$ and $v \geq 7$ 
(see Theorem~18.4 
of~\cite{CR}),
and there exists a $4$-chromatic STS$(v)$ if and only if $v\equiv 1$ or $3\pmod{6}$ and $v\geq 21$~\cite{CFGGKOPP,Haddad}.
It is shown in~\cite{dBPR} that for each integer $\delta\geq 3$ and sufficiently large $v\equiv 1$ or $3\pmod{6}$, there exists an STS$(v)$
with chromatic number $\delta$;
this result was generalized to equitable 
$\delta$-colourings in~\cite{CHL}.
Examples of Steiner triple systems which have chromatic number $3$ but do not have equitable colourings are given in~\cite{FGG}.
For triple systems with $\lambda \geq 2$, it was shown in~\cite{HP2014} that for each integer $\delta \geq 3$
there is a $\delta$-chromatic $(v,3,\lambda)$-BIBD for all sufficiently large admissible $v$;
more generally, for all $\delta \geq2$ and $k\geq 3$ other than $(\delta,k)$=(3,2),
it was proved that the necessary conditions for the existence of a $(v,k,\lambda)$-BIBD are
asymptotically sufficient for the existence of a $\delta$-chromatic $(v,k,\lambda)$-BIBD.

A {\em parallel class} of a
BIBD
$(V, {\mathcal B})$
is a subset of ${\mathcal B}$ which includes each element of
$V$ exactly once; that is, a parallel class is a partition of $V$ into blocks of ${\mathcal B}$.
A BIBD $(V, {\mathcal B})$
is said to be {\em resolvable}
if ${\mathcal B}$  partitions into parallel classes.
Vertex colourings of resolvable designs appear to be first mentioned in~\cite{LusiColbourn2023},
where it is observed that a family of 
$3$-colourable resolvable triple systems is produced by the Bose construction.

In this paper we focus directly on the problem of 
colouring resolvable Steiner triple systems.
A resolvable $(v,3,1)$-BIBD together with a partition into parallel classes
is known as a {\em Kirkman triple system},  or KTS$(v)$.
A KTS$(v)$ exists if and only if $v\equiv 3\pmod{6}$~\cite{Lu1965,RCW1971};
for each such order $v$ we prove that there exists a $3$-chromatic KTS$(v)$.
We also prove that there are infinitely many $\delta$-chromatic Kirkman triple systems for each integer $\delta \geq 3$.

We say that a colouring $\phi$ of a KTS$(v)$ is {\em rainbow} if it is a 
$3$-colouring in which the colour class sizes are each $v/3$ and there exists a parallel class, which we call a {\em rainbow parallel class}, in which every triple receives all three colours.  A straightforward counting argument yields that in a rainbow colouring there is exactly one rainbow parallel class, as any block not in the rainbow parallel class of a rainbow colouring of a KTS receives exactly two distinct colours.
We may also refer to a KTS$(v)$ as {\em rainbow} if it possesses a colouring which is rainbow.

Our aim in this paper is to construct 
colourings for Kirkman triple systems.
We first give results on small orders in Section~\ref{smallorders}.
In Section~\ref{mainconstructions}, we manipulate standard constructions of Kirkman triple systems by incorporating rainbow colourings.  These methods involve group divisible designs.
\begin{definition}
A $k$-GDD or {\em Group Divisible Design} is a
 triple $(V,{\mathcal G},{\mathcal B})$ where
${\mathcal G}$ is a partition of $V$ into
subsets of points, called {\em groups}, and
${\mathcal B}$ is a set of {\em blocks}, each of size $k$, with the following properties:
\begin{itemize}
	\item every pair of points in different groups belongs to exactly one block of ${\mathcal B}$;
	\item no block intersects a group in more than one point.
\end{itemize}
\end{definition}
\noindent
The {\em group type}, or just {\em type}, of a $k$-GDD is the multiset $T=\{|G|\mid G\in {\mathcal G}\}$.
If the integer $g$ occurs $u$ times in $T$ we use the notation $g^u$.
Similarly to BIBDs, a $k$-GDD is said to be {\em resolvable} if the block set $\cal B$ can be partitioned into parallel classes, in which case we refer to a $k$-RGDD.
In particular, we explore constructions using a class of GDDs known as {\em frames} and prove the following.

\begin{theorem}
\label{Thm-main1}
For each $v\equiv 3\pmod{6}$ there exists a rainbow {\rm KTS$(v)$} with chromatic number $3$.
\end{theorem}

We apply similar techniques to study $4$-colourings in Section~\ref{thefourcase}.
Our main result in that section is Theorem \ref{main4colour}, in which we show the existence of an infinite
class of 
$4$-colourable Kirkman triple systems
for which the orders grow according to a linear function.
In Section~\ref{arbitrary}, we study the case where there is an arbitrary number of colours, giving a proof of the following theorem.

\begin{theorem}
\label{Thm-main2}
For each $\delta\geq 3$, there are infinitely many integers $v$ such that there exists a {\rm KTS$(v)$} with chromatic number $\delta$.
\end{theorem}

A simple way to construct Kirkman triple systems is to use quadruple systems.
We define $Q(v)$ to be a {\em quadruple system of order $v$}, namely a $(v,4,1)$-BIBD.
A quadruple system $Q(v)$ is known to exist if and only if $v \equiv 1$ or $4 \pmod{12}$~\cite{Hanani1961}.
Using quadruple system techniques, we discover interesting relationships between colourability of Kirkman triple systems and their underlying quadruple systems, which we describe in Section~\ref{quadruple}. We conclude the paper in Section~\ref{conclusion} by discussing some open problems.

\section{Results for small orders}
\label{smallorders}

It is well known~\cite{CR} that there
is a unique KTS$(9)$ up to isomorphism, shown by the following triples on the point set $\{1,2,\ldots,9\}$.
$$
\begin{array}{cccc}
\{1,2,3\} & \{1,4,7\} & \{1,5,9\} & \{1,6,8\} \\
\{4,5,6\} & \{2,5,8\} & \{2,6,7\} & \{2,4,9\} \\
\{7,8,9\} & \{3,6,9\} & \{3,4,8\} & \{3,5,7\} \\
\end{array}
$$

\begin{lemma}
\label{Lemma-kts9colourings}
There exist 
$3$-colourings of {\rm KTS}$(9)$ for each of the following colour types: $3^3$, $2^13^14^1$ and  $1^14^2$.
\end{lemma}

\begin{proof}
Using the representation of the KTS(9) above, a 
$3$-colouring with colour type $3^3$ is given by the colour classes $\{1,4,9\}$, $\{2,5,7\}$ and $\{3,6,8\}$;
this is a rainbow colouring for which the parallel class containing the block $\{1,2,3\}$ is rainbow.
For colour type $2^13^14^1$, consider the classes $\{1,2\}$, $\{3,5,8\}$ and $\{4,6,7,9\}$.
For colour type $1^14^2$, consider the classes $\{1\}$, $\{2,3,5,9\}$ and $\{4,6,7,8\}$.
\end{proof}

Now, for several small orders we explicitly provide rainbow colourings
that will be useful in subsequent constructions.
	
\begin{lemma}
	There exists a rainbow {\rm KTS}$(v)$ for each $v\in \{3,9,15,21,33,39,57,69\}$.
	\label{Lemma-smallrainbows}
\end{lemma}

\begin{proof}
The result is clearly true for $v=3$, and for $v=9$ a rainbow colouring appears in the proof of Lemma~\ref{Lemma-kts9colourings}.
For $v=15$, the following KTS$(15)$ (taken from Example 5.1.1 in~\cite{LindnerRodger}) has a rainbow colouring under the point partition $\{1,4,7,13,14\}$,
		$\{2,5,6,8,9\}$ and
		$\{3,10,11,12,15\}$.
\[
\begin{array}{ccccccc}
\{1,2,3\} & \{1,4,5\} & \{1,6,7\} & \{1,8,9\} & \{1,10,11\} & \{1,12,13\} & \{1,14,15\} \\
\{4,8,12\} & \{2,8,10\} & \{2,9,11\} & \{2,12,15\} & \{2,13,14\} & \{2,4,6\} & \{2,5,7\} \\
\{5,10,14\} & \{3,13,15\} & \{3,12,14\} & \{3,5,6\} & \{3,4,7\} & \{3,9,10\} & \{3,8,11\} \\
\{6,11,13\} & \{6,9,14\} & \{4,10,15\} & \{4,11,14\} & \{5,9,12\} & \{5,11,15\} & \{4,9,13\} \\
\{7,9,15\} & \{7,11,12\} & \{5,8,13\} & \{7,10,13\} & \{6,8,15\} & \{7,8,14\} & \{6,10,12\} \\
\end{array}
\]

For $v \in \{21,33,39,57,69\}$ we constructed several KTS$(v)$ in a manner similar to that described in~\cite{ColbournLing2002},
and in so doing we were able to identify several systems with rainbow
\linebreak
colourings.
The general approach here is to construct a KTS$(v)$ on the point set
\linebreak
$\{i_j\mid i\in
{\mathbb Z}_{(v-3)/3},
j\in {\mathbb Z}_3
\}\cup \{\infty_0,\infty_1,\infty_2\}$
such that the permutation
\begin{center}
$\sigma=(0_0, 1_0, \ldots, (\frac{v-6}{3})_0)(0_1, 1_1, \ldots, (\frac{v-6}{3})_1)(0_2, 1_2, \ldots, (\frac{v-6}{3})_2)(\infty_0)(\infty_1)(\infty_2)$
\end{center}
fixes one parallel class,
induces a short orbit of $\frac{v-3}{6}$ parallel classes
and induces a long orbit of $\frac{v-3}{3}$ parallel classes.
By assigning colour $j$ to each point $i_j$ where $i \in {\mathbb Z}_{(v-3)/3}\cup \{\infty\}$,
we will obtain a 
$3$-colouring for which the parallel class fixed by $\sigma$ is rainbow.

For $v=21$, the fixed parallel class contains the block $\{\infty_0,\infty_1,\infty_2\}$ as well as the development of the starter block
$\{0_0, 4_1, 1_2\}$.
The short orbit of parallel classes is generated from a parallel class having the blocks
\begin{center}
		$\{1_0,1_1,2_1\}$,
 	 $\{2_0,1_2,2_2\}$,
		$\sigma^{3}(\{1_0,1_1,2_1\})$,
	$\sigma^{3}(\{2_0,1_2,2_2\})$, \\
	$\{\infty_0,0_2,3_2\}$,
	$\{\infty_1,0_0,3_0\}$,
	$\{\infty_2,0_1,3_1\}$.
	\end{center}
The long orbit of parallel classes is generated from a parallel class having blocks
\begin{center}
$\{0_0,1_0,3_1\}$,
$\{2_0,4_0,0_2\}$,
$\{1_1,5_1,3_2\}$,
$\{0_1,1_2,5_2\}$, \\
$\{\infty_1,4_1,4_2\}$,
$\{\infty_2,5_0,2_2\}$,
$\{\infty_0,3_0,2_1\}$.
\end{center}

For $v=33$, the fixed parallel class contains the block $\{\infty_0,\infty_1,\infty_2\}$ as well as the development of the starter block
$\{0_0, 5_1, 2_2\}$.
The short orbit of parallel classes is generated from a parallel class having the blocks
\begin{center}
$\{1_0, 2_0, 1_2\}$,  $\{3_0, 1_1, 2_1\}$,  $\{4_0, 2_2, 8_2\}$,  $\{3_1, 9_1, 4_2\}$,
\\
$\sigma^{5}(\{1_0, 2_0, 1_2\})$,  $\sigma^{5}(\{3_0, 1_1, 2_1\})$,  $\sigma^{5}(\{4_0, 2_2, 8_2\})$,  $\sigma^{5}(\{3_1, 9_1, 4_2\})$,
\\
$\{\infty_1, 0_0, 5_0\}$,
$\{\infty_2, 0_1, 5_1\}$,
$\{\infty_0, 0_2, 5_2\}$.
\end{center}
The long orbit of parallel classes is generated from a parallel class having blocks
\begin{center}
$\{0_0, 2_0, 2_1\}$,  $\{1_0, 4_0, 7_1\}$,  $\{3_0, 7_0, 4_1\}$,  $\{5_0, 6_2, 8_2\}$,  $\{8_0, 4_2, 5_2\}$,  $\{0_1, 8_1, 2_2\}$,
\\
$\{6_1, 9_1, 9_2\}$,
$\{1_1, 0_2, 7_2\}$,
$\{\infty_1, 5_1, 3_2\}$,  $\{\infty_2, 6_0, 1_2\}$,  $\{\infty_0, 9_0, 3_1\}$.
\end{center}

For $v=39$, the fixed parallel class contains the block $\{\infty_0,\infty_1,\infty_2\}$ plus the development of the starter block
$\{0_0, 9_1, 3_2\}$.
The short orbit of parallel classes is generated from a parallel class having blocks
\begin{center}
$\{1_0, 2_0, 1_1\}$,  $\{3_0, 4_1, 5_1\}$,  $\{4_0, 1_2, 2_2\}$,  $\{5_0, 4_2, 9_2\}$,  $\{2_1, 9_1, 5_2\}$,
\\
$\sigma^{6}(\{1_0, 2_0, 1_1\})$,  $\sigma^{6}(\{3_0, 4_1, 5_1\})$,  $\sigma^{6}(\{4_0, 1_2, 2_2\})$,  $\sigma^{6}(\{5_0, 4_2, 9_2\})$,
\\
$\sigma^{6}(\{2_1, 9_1, 5_2\})$,
$\{\infty_1, 0_0, 6_0\}$,
$\{\infty_2, 0_1, 6_1\}$,
$\{\infty_0, 0_2, 6_2\}$.
\end{center}
The long orbit of parallel classes is generated from a parallel class having blocks
\begin{center}
$\{0_0, 2_0, 5_1\}$,  $\{1_0, 4_0, 8_1\}$,  $\{3_0, 7_0, 1_1\}$,  $\{5_0, 10_0, 0_2\}$,  $\{6_0, 7_2, 11_2\}$,  $\{9_0, 3_2, 5_2\}$,
$\{0_1, 9_1, 2_2\}$,  $\{2_1, 4_1, 1_2\}$,  $\{6_1, 10_1, 10_2\}$,  $\{11_1, 6_2, 9_2\}$,
$\{\infty_1, 3_1, 4_2\}$,  $\{\infty_2, 8_0, 8_2\}$,  $\{\infty_0, 11_0, 7_1\}$.
\end{center}

For $v=57$, the fixed parallel class contains the block $\{\infty_0,\infty_1,\infty_2\}$ plus the development of the starter block
$\{0_0, 9_1, 10_2\}$.
The short orbit of parallel classes is generated from a parallel class having blocks
\begin{center}
$\{1_0, 2_0, 1_1\}$,  $\{3_0, 5_0, 1_2\}$,  $\{4_0, 2_1, 5_1\}$,  $\{6_0, 3_1, 8_1\}$,  $\{7_0, 2_2, 4_2\}$,  $\{8_0, 7_2, 8_2\}$,  $\{6_1, 7_1, 3_2\}$,  $\{13_1, 5_2, 15_2\}$,
$\sigma^{9}(\{1_0, 2_0, 1_1\})$,  $\sigma^{9}(\{3_0, 5_0, 1_2\})$,  $\sigma^{9}(\{4_0, 2_1, 5_1\})$,  $\sigma^{9}(\{6_0, 3_1, 8_1\})$,  $\sigma^{9}(\{7_0, 2_2, 4_2\})$,  $\sigma^{9}(\{8_0, 7_2, 8_2\})$,  $\sigma^{9}(\{6_1, 7_1, 3_2\})$,  $\sigma^{9}(\{13_1, 5_2, 15_2\})$,
$\{\infty_1, 0_0, 9_0\}$,
$\{\infty_2, 0_1, 9_1\}$,
$\{\infty_0, 0_2, 9_2\}$.
\end{center}
The long orbit of parallel classes is generated from a parallel class having blocks
\begin{center}
$\{0_0, 3_0, 6_1\}$,
$\{1_0, 5_0, 9_1\}$,
$\{2_0, 8_0, 15_1\}$,
$\{4_0, 11_0, 13_2\}$,
$\{6_0, 16_0, 9_2\}$,
$\{12_0, 17_0, 6_2\}$,
$\{7_0, 0_1, 12_1\}$,
$\{9_0, 1_1, 3_1\}$,
$\{13_0, 1_2, 14_2\}$,
$\{14_0, 0_2, 4_2\}$,
$\{7_1, 14_1, 7_2\}$,
$\{8_1, 16_1, 2_2\}$,
$\{13_1, 17_1, 16_2\}$,
$\{2_1, 8_2, 11_2\}$,
$\{5_1, 3_2, 10_2\}$,
$\{10_1, 5_2, 17_2\}$,
$\{\infty_1, 4_1, 12_2\}$,
$\{\infty_2, 10_0, 15_2\}$,
$\{\infty_0, 15_0, 11_1\}$.
\end{center}

For $v=69$, the fixed parallel class contains the block $\{\infty_0,\infty_1,\infty_2\}$ plus the development of the starter block
$\{0_0, 14_1, 9_2\}$.
The short orbit of parallel classes is generated from a parallel class having blocks
\begin{center}
$\{1_0, 2_0, 1_1\}$,
$\{3_0, 5_0, 6_1\}$,
$\{4_0, 7_0, 1_2\}$,
$\{6_0, 2_1, 3_1\}$,
$\{8_0, 3_2, 4_2\}$,
$\{9_0, 2_2, 7_2\}$,
$\{10_0, 9_2, 16_2\}$,
$\{9_1, 15_1, 6_2\}$,
$\{10_1, 18_1, 8_2\}$,
$\{16_1, 19_1, 21_2\}$,
$\sigma^{11}(\{1_0, 2_0, 1_1\})$,
$\sigma^{11}(\{3_0, 5_0, 6_1\})$,
$\sigma^{11}(\{4_0, 7_0, 1_2\})$,
$\sigma^{11}(\{6_0, 2_1, 3_1\})$,
$\sigma^{11}(\{8_0, 3_2, 4_2\})$,
$\sigma^{11}(\{9_0, 2_2, 7_2\})$,
$\sigma^{11}(\{10_0, 9_2, 16_2\})$,
$\sigma^{11}(\{9_1, 15_1, 6_2\})$,
$\sigma^{11}(\{10_1, 18_1, 8_2\})$,
$\sigma^{11}(\{16_1, 19_1, 21_2\})$,
$\{\infty_1, 0_0, 11_0\}$,
$\{\infty_2, 0_1, 11_1\}$,
$\{\infty_0, 0_2, 11_2\}$.
\end{center}
The long orbit of parallel classes is generated from a parallel class having blocks
\begin{center}
$\{0_0, 4_0, 2_1\}$,
$\{1_0, 6_0, 10_1\}$,
$\{3_0, 9_0, 14_1\}$,
$\{5_0, 20_0, 15_1\}$,
$\{7_0, 17_0, 8_2\}$,
$\{11_0, 19_0, 19_2\}$,
$\{12_0, 21_0, 1_2\}$,
$\{14_0, 4_1, 8_1\}$,
$\{15_0, 1_1, 6_1\}$,
$\{16_0, 0_1, 9_1\}$,
$\{8_0, 13_2, 15_2\}$,
$\{10_0, 2_2, 14_2\}$,
$\{13_0, 3_2, 16_2\}$,
$\{11_1, 21_1, 7_2\}$,
$\{13_1, 20_1, 20_2\}$,
$\{17_1, 19_1, 11_2\}$,
$\{5_1, 6_2, 9_2\}$,
$\{7_1, 10_2, 18_2\}$,
$\{12_1, 5_2, 21_2\}$,
$\{16_1, 0_2, 4_2\}$,
$\{\infty_1, 18_1, 17_2\}$,
$\{\infty_2, 2_0, 12_2\}$,
$\{\infty_0, 18_0, 3_1\}$.
\end{center}
\end{proof}

We will discuss rainbow colourings in more detail in Section~\ref{mainconstructions},
but we pause here to briefly note that
the several preceding instances of $3$-chromatic Kirkman triple systems can be used to generate
further examples of $3$-chromatic Kirkman triple systems.

\begin{theorem}
\label{Thm-tripling}
	Let $v\geq 3$.
	Suppose that there exists a {\rm KTS}$(3v)$ which has a $3$-colouring that is 
	equitable.
	Then there exists a {\rm KTS}$(9v)$
	which has a $3$-colouring that is 
	equitable.
\end{theorem}

\begin{proof}

	Let $A_1=\{a_1,a_2,\dots ,a_{v}\}$,
	$A_2= \{a_{v+1},a_{v+2},\dots ,a_{2v}\}$ and
	$A_3= \{a_{2v+1},a_{2v+2},\dots ,a_{3v}\}$,
	with $B_1$, $B_2$, $B_3$, $C_1$, $C_2$ and $C_3$ similarly defined.
	Next, let $A=A_1\cup A_2\cup A_3$ with $B$ and $C$ similarly defined.

		We construct a KTS$(9v)= (A\cup B\cup C, {\mathcal B})$ as follows.
		Add to ${\mathcal B}$ the blocks of a 
		$3$-coloured KTS$(3v)$ on each of $A$, $B$ and $C$ in such a way that the vertices in the set $A_i\cup B_i\cup C_i$ have colour $i$,
	for each $1\leq i\leq 3$.
	
	Necessarily $v$ is odd; otherwise a KTS$(3v)$ does not exist.
	Thus there exists a Latin square $L_v$ of order $v$ which partitions into transversals $T_1,T_2,\dots ,T_v$.
	Equivalently, there is a 3-RGDD of type $v^3$ (a resolvable transversal design), $G$.
		Add to ${\mathcal B}$ 
	 nine copies of $G$, using the following sets of three groups:

\[\begin{array}{ccc}
	\{A_1,B_1,C_2\}, & \{A_2,B_2,C_3\}, & \{A_3,B_3,C_1\}, \\
	\{A_1,B_2,C_1\}, &  \{A_2,B_3,C_2\}, & \{A_3,B_1,C_3\}, \\
	\{A_1,B_3,C_3\}, & \{A_2,B_1,C_1\}, & \{A_3,B_2,C_2\}.\\
\end{array}\]

To see that the added blocks can be partitioned into parallel classes, let
\[\begin{array}{l}
\mathcal{X}_1 = \big( \{A_1,B_1,C_2\}, \{A_2,B_2,C_3\}, \{A_3,B_3,C_1\} \big), \\
\mathcal{X}_2 = \big( \{A_1,B_2,C_1\}, \{A_2,B_3,C_2\}, \{A_3,B_1,C_3\} \big), \\
\mathcal{X}_3 = \big( \{A_1,B_3,C_3\}, \{A_2,B_1,C_1\}, \{A_3,B_2,C_2\} \big),
\end{array}\] 
and let $G_{i,j,k}$ be the $k^{\mathrm{th}}$ parallel class in the copy of $G$ placed on the $j^{\mathrm{th}}$ set of groups in $\mathcal{X}_i$.  For each $i \in \{1,2,3\}$ and $k \in \{1, \ldots, v\}$, $G_{i,1,k} \cup G_{i,2,k} \cup G_{i,3,k}$ forms a parallel class. 
\end{proof}

In addition to the Kirkman triple systems produced by Theorem~\ref{Thm-tripling}, we call attention to
the family of
Kirkman triple systems that arise from the Bose construction.
Resolvable Bose triple systems are studied in~\cite{LusiColbourn2023},
where it is shown that for such a triple system to exist it must have order $v \equiv 9\pmod{18}$.
Moreover, for each $v \equiv 9\pmod{18}$ such a system is constructed.
The resolvable Bose triple systems that are constructed in~\cite{LusiColbourn2023} are observed to all have equitable 
$3$-colourings.

We finish this section with some comments on small Kirkman triple systems that are not $3$-chromatic.
In~\cite{TV}, Tonchev and Vanstone presented thirty non-isomorphic KTS$(33)$.  
We analysed each of them computationally by implementing an exhaustive backtracking algorithm to search for potential colourings.
In so doing we found that none of them can be $3$-coloured, but they each admit a 
$4$-colouring that is equitable and of type $8^39^1$;
we include an appendix with details of such colourings.
Additionally, we tested the 1-rotational KTS$(33)$ over the cyclic group given by Buratti and Zuanni~\cite{BurattiZuanni}, and found that 13 are $3$-chromatic, while the rest are all $4$-chromatic; we give an example of one of the $3$-chromatic systems from~\cite{BurattiZuanni} in the appendix.
We also tested the 3-pyramidal KTS(33) from \cite{Bonvicini} and found that it was 4-chromatic.
It seems that, at least in the case $v=33$, a high degree of symmetry correlates with a higher chromatic number. 

Lastly we observe that sporadic examples of Kirkman triple systems with chromatic number greater than $4$ appear in the literature.  Fug\`{e}re, Haddad and Wehlau~\cite{FHW1994} prove that the points and lines of the projective geometry PG$(5,2)$ form a $5$-chromatic Steiner triple system; it is known that PG$(n,2)$ is resolvable whenever $n$ is odd~\cite{Baker1976}, so that this design yields a $5$-chromatic KTS$(63)$.
In~\cite{Haddad}, Haddad shows that the Steiner triple system derived from the affine geometry AG$(4,3)$ is $5$-chromatic; it follows that there exists a KTS$(81)$ which is $5$-chromatic. Bruen, Haddad and Wehlau~\cite{BHW1998} show that the KTS$(243)$ given by the affine geometry AG$(5,3)$ is $6$-chromatic.

\section{Constructions via frames}

\label{mainconstructions}

The main aim of this section is to prove Theorem  \ref{Thm-main1}.
We begin with some definitions.
Recall that a $k$-GDD consists of a partition $\mathcal{G}$ of the point set $V$ into groups, and a collection $\mathcal{B}$ of blocks of size $k$ which
contain each pair of points in different groups exactly once, but no pair of points contained in a single group.
A $k$-GDD is said, in turn, to be a {\em $k$-frame} if
there is a partition ${\mathcal P}$ of ${\mathcal B}$ into {\em partial parallel classes}
such that each element of ${\mathcal P}$ is a partition of $V\setminus G$ for some group
$G\in {\mathcal G}$.
\noindent
A frame thus can be thought of as a  partially resolvable group divisible design, where each partial parallel class misses exactly one group.
A $3$-frame is also known as a {\em Kirkman frame}.
\begin{remark} \label{FPC}
By a standard straightforward counting argument, in a Kirkman frame, for each $G\in {\mathcal G}$ there are $|G|/2$ partial parallel classes in ${\mathcal P}$ which exclude $G$; hence, every group has even size.
\end{remark}

\begin{theorem}[\cite{Stinson87}]
	There is a Kirkman frame of type $g^u$ if and only if $u\geq 4$, $g$ is even and $g(u-1)$ is divisible by $3$.
\label{framebyanyothername}
\end{theorem}

Kirkman frames have also been classified when all but one group has uniform size, and the group sizes are either small or divisible by $12$.

\begin{theorem}
	[\cite{GRS, WG}]
	Let $0<g\leq 12$ or $g\equiv 0\pmod{12}$ and $u,m>0$. Then there exists a Kirkman frame of type $g^u m^1$ if and only if the following all hold:
	\begin{enumerate}
		\item[a)] $g$ is even;
		\item[b)] $gu \equiv 0\pmod{3}$;
		\item[c)] $m\equiv g \pmod{6}$;
		\item[d)] $u\geq 3$;
		\item[e)] if $u = 3$ then $m = g$, and
		if $u > 3$ then $0 < m \leq g(u - 1)/2$.
	\end{enumerate}
\label{Thm-ivebeenframed}
\end{theorem}

A {\em subsystem} of a Kirkman triple system $(V,{\mathcal B})$ with parallel classes ${\mathcal R}$ is a Kirkman triple system
$(V',{\mathcal B}')$ with parallel classes ${\mathcal R}'$ such that
$V'\subseteq V$,
${\mathcal B}'\subseteq {\mathcal B}$, and
each element of ${\mathcal R}'$ is a subset of some
element of ${\mathcal R}$.
Equivalently, the parallel classes of ${\mathcal R}'$ are those of ${\mathcal R}$ induced by restricting to the subset of points $V'$.
Note that each triple of ${\mathcal B}$ intersects $V'$ in either $0$, $1$ or $3$ points.

A {\em rainbow subsystem} of a rainbow Kirkman triple system with colouring $\phi$ is a subsystem which is also rainbow under the colouring $\phi$ restricted to the points of the subsystem.
Note that any rainbow Kirkman triple system trivially contains a rainbow subsystem of order $3$.
Also note that every rainbow Kirkman triple system is $3$-chromatic, save for the degenerate case of a KTS$(3)$, which is $2$-chromatic.

Kirkman frames are a key tool for constructing Kirkman triple systems, particularly Kirkman triple systems with subsystems. The standard methods are summarized in \cite{Stinson87}. These methods are here manipulated to obtain the proof of the following theorem, which may be regarded
as the main construction in this paper.

\begin{theorem}
	\label{Thm-RainbowFrame}
	Suppose there exists a Kirkman frame on $v$ points of type $T$.
	Suppose that for some positive integer $w$,  for each $t\in T$ there exists a {\rm KTS}$(3t+w)$ with a colouring $\phi$ which is rainbow and containing a rainbow subsystem on $w$ points.
	Then there exists a rainbow {\rm KTS}$(3v+w)$.
\end{theorem}

\begin{proof}
	Let	$(V,{\mathcal G},{\mathcal B})$ be a Kirkman frame of type $T$; note that necessarily this means that every $t\in T$ is even (see Remark~\ref{FPC}).
	Further, the existence of a KTS$(3t+w)$ implies that $w \equiv 3 \pmod{6}$,
	so we let $w=3w'$, where $w'$ is odd.
	Let $W=\{\infty_i\mid i\in \{1,2,\ldots,w'\}\}$, $V'=V\times \{0,1,2\}$ and let $W'=W\times \{0,1,2\}$.
	 	We will construct a rainbow KTS$(3v+w)$, $(V'\cup W',{\mathcal B}'')$
	with a colouring $\phi'$ which maps $(x,i)$ to colour $i$ for each $x\in V\cup W$ and $i\in \{0,1,2\}$.
	
	For each $B=\{x,y,z\}\in {\mathcal B}$,
	construct the following parallel classes of blocks on the set of points $B\times \{0,1,2\}$:
	\begin{itemize}
		\item  $\{\{(x,0),(y,0),(z,1)\}, \{(x,1),(y,1),(z,2)\},\{(x,2),(y,2),(z,0)\}\}$;
		\item  $\{\{(x,0),(y,2),(z,2)\}, \{(x,1),(y,0),(z,0)\},\{(x,2),(y,1),(z,1)\}\}$;
		\item  $\{\{(x,0),(y,1),(z,0)\}, \{(x,1),(y,2),(z,1)\},\{(x,2),(y,0),(z,2)\}\}$.
	\end{itemize}
	Let the resultant set of blocks be ${\mathcal B}'$.
	Let ${\mathcal G}'=\{G\times \{0,1,2\} \mid G\in {\mathcal G}\}$.
	Note that the structure $(V',{\mathcal G}',{\mathcal B}')$ is a Kirkman frame of type $\{3t\mid t\in T\}$. Let ${\mathcal P}'$ be the associated set of partial parallel classes and note that by Remark~\ref{FPC}, each group $G'\in {\cal G}'$ of size $3t$ is missed by $3t/2$ partial classes.
	Observe that each block of ${\mathcal B}'$ receives exactly two distinct colours under the colouring $\phi'$.

	We fix a group $G_0\in {\cal G}'$ of size $3t_0$ and 
		add to ${\mathcal B}''$ a copy of the KTS$(3t_0+w)$, $K_0$, on $G_0\cup W'$ so that the subsystem of order $w$ is on the points of $W'$.  Let ${\mathcal R}_1(G_0)$ be the set of
	$(w-1)/2$ parallel classes in $K_0$ which each include a partition of the points of $W'$.  That is, blocks in each of these parallel classes intersect $W$ in either $0$ or $3$ points.
	 Let
	${\mathcal R}_2(G_0)$ be the set of remaining parallel classes, where
	$|{\mathcal R}_2(G_0)|=(3t_0+w-1)/2-|{\mathcal R}_1(G_0)|=3t_0/2$.
Further, we
	add to ${\mathcal B}''$
the blocks of the rainbow class so that they are of the form	$\{(x,0), (y,1), (z,2)\}$,  consistent with the colouring $\phi'$.

	For every other group $G \in {\cal G}'\setminus \{G_0\}$ of size $3t$, we remove the blocks of the subsystem of order $w$ from the KTS$(3t+w)$, creating a KTS$(3t+w)$ with a hole of size $w$, $K_1(G)$.
	We
		add to ${\mathcal B}''$
	 a copy of $K_1(G)$ on $G\cup W'$, so that the hole is on the points of $W'$.  Further, we again place the blocks of the rainbow class so that they are of the form	$\{(x,0), (y,1), (z,2)\}$, consistent with the colouring $\phi'$.
Let ${\mathcal R}_1(G)$ (resp.\ ${\mathcal R}_2(G)$) denote the set of ``short'' (resp.\ ``full'') classes of $K_1(G)$.
Note that the $|{\mathcal R}_1(G)|=(w-1)/2$ ``short'' classes cover $G$,
and each of the $|{\mathcal R}_2(G)|=(3t+w-1)/2-|{\mathcal R}_1(G)|=3t/2$ ``full'' classes covers $G\cup W'$.

	Observe that
	$(V'\cup W',{\mathcal B}'')$
	 defined above
	  is
	 a Steiner triple system, 
	 and furthermore that $\phi'$ is a 
	$3$-colouring of this Steiner triple system.
	It remains to show that this system is also a Kirkman triple system with a rainbow parallel class.

	For each group $G\in {\mathcal G'}$, by Remark~\ref{FPC}, there exist $|G|/2$ partial parallel classes $P\in {\mathcal P'}$ such that $P$ excludes the points in $G$.
	Match these with the $|G|/2$
	parallel classes of
    ${\mathcal R}_2(G)$
    on $G\cup W'$
	to create $|G|/2$ parallel classes on $V'$.
	
	Next, we create the rainbow parallel class by taking the rainbow parallel class from ${\mathcal R}_1(G_0)$ (placed on  $G_0\cup W'$) together with the rainbow parallel class from ${\mathcal R}_1(G)$ for each $G\neq G_0$.
	Finally, create $(w-3)/2$ more parallel classes by
	combining the non-rainbow parallel classes from 	${\mathcal R}_1(G_0)$
	with the non-rainbow parallel classes from
	${\mathcal R}_1(G)$ for each $G\neq G_0$.
	\end{proof}

We are now ready to prove Theorem~\ref{Thm-main1}, that there is a $3$-chromatic
rainbow KTS$(v)$ for every $v\equiv 3\pmod{6}$.

\begin{proof}[{\it Proof of Theorem \ref{Thm-main1}}]
The cases $v\in X = \{3, 9, 15, 21, 33, 39, 57, 69\}$ are dealt with in
Lemma~\ref{Lemma-smallrainbows};
 therefore we assume $v\not\in X$. Set
 $\omega=(v-3)/3$.
 If $\omega\not\equiv 10\pmod{12}$, let
$g \in \{2, 4, 6\}$ such that $g \equiv \omega\pmod{6}$ and set $(m, u) = (g, \omega/g-1)$. Otherwise,
set $(g, m, u) = (4, 10, (\omega-10)/4)$.
 Since $v\not\in X$, one can check that $(g, m, u)$
satisfies
Theorem~\ref{Thm-ivebeenframed}: in particular, $u\geq 3$, and $u \neq 3$ when $m \neq g$ (that
is, when $\omega\equiv 10 \pmod{12}$. Therefore, there exists a Kirkman frame of
type $g^um^1$. Also, by
Lemma~\ref{Lemma-smallrainbows}, there exists a rainbow KTS$(3m + 3)$, for
$m = 2, 4, 6, 10$. Then, by
 Theorem~\ref{Thm-RainbowFrame},
 there exists a rainbow KTS$(3\omega +3)$.
\end{proof}

\section{\boldmath{$4$}-chromatic Kirkman triple systems}

\label{thefourcase}

We further exploit the standard frame methods (see \cite{Stinson87})  to obtain results on $4$-chromatic Kirkman triple systems.

\begin{lemma}
	Suppose that there exists a $4$-coloured (not necessarily $4$-chromatic) $4$-{\rm GDD} of type $4^u$ such that:
	\renewcommand{\labelenumi}{\alph{enumi})}
	\begin{enumerate}
		\item
		 in any block there are at most two points of the same colour; and
		 \item
		in any group there is exactly one point of each colour.
	\end{enumerate}
	Then, for any $g\geq 1$, there exists a 
	$4$-coloured $3$-frame of type $(8g)^u$ such that each group is equitably coloured.
\label{blowinup}
\end{lemma}

\begin{proof}
Let $(V, {\mathcal G}, {\mathcal B})$ be a $4$-GDD of type $4^u$ satisfying the conditions of the lemma.
We construct a 
$4$-coloured $3$-frame $(V^{\ast},{\mathcal G}^{\ast}, {\mathcal B}^{\ast})$ of type $(8g)^u$ as follows.
The point set will be
\[V^{\ast}= V\times\{1,2,\ldots,2g\} = \{(v, i)\mid v\in V, i\in \{1,2,\ldots,2g\}\}.\]
For each $v \in V$ and $i \in \{1,2,\ldots,2g\}$, the vertex $(v,i)$ is coloured with the same colour as the vertex $v$ in the original $4$-GDD.
The set of groups will be
$${\mathcal G}^{\ast}
= \{ G \times \{1,2,\ldots,2g\} \mid G \in \mathcal{G}\}.
$$
Observe that in each group, there are exactly $2g$ points of each colour, as required by the conditions of this lemma.

We next describe the block set ${\mathcal B}^{\ast}$.
By Theorem  \ref{framebyanyothername}, there exists a $3$-frame
 of type $(2g)^4$.
For each $B\in {\mathcal B}$, place a copy, $F(B)$, of this $3$-frame on the point set
\linebreak
$\{(v,i) \mid v\in B, i\in \{1,2,\ldots,2g\}\}$ so that the above groups are preserved.
By property (a), all the triples thus constructed are 
$4$-coloured.

The fact that the resultant  $(V^{\ast},{\mathcal G}^{\ast},{\mathcal B}^{\ast})$ is a $3$-frame follows by the Fundamental Frame Construction (Construction 3.1 from \cite{Stinson87}, see also~\cite{Stinson91}).
To see this, consider the original $4$-GDD of type $4^u$. Let 
$v\in V$, let $G_v$ be the group of ${\mathcal G}$ containing $v$ and let 
${\mathcal B}_v\subset {\mathcal B}$ be the set of blocks (each of size $4$) which contain $v$.  Next, by the definition of a GDD, the set of triples 
$\{B\setminus \{v\}\mid v\in {\mathcal B}_v \}$ 
is a partition of $V\setminus G_v$. 
It follows that the set of $3$-frames of the form $\{F(B)\mid B\in {\mathcal B}_v\}$ partition into 
$g$ partial parallel classes of $V^{\ast}$, each of which exclude $G_v \times \{1,2,\dots ,2g\}$. 
Considering each vertex of $v\in V$ in turn,  ${\mathcal B}^{\ast}$ partitions into the required partial parallel classes.  
\end{proof}

\begin{lemma}
	Let $u\equiv 1$ or $4 \pmod{12}$ and $u\geq 4$.
	Then there exists a $4$-coloured
	(not necessarily $4$-chromatic)
 	$4$-{\rm GDD} of type $4^u$ with block size $4$ such that:
	\renewcommand{\labelenumi}{\alph{enumi})}
	\begin{enumerate}
		\item in any block there are at most two points of the same colour; and
		\item in any group there is exactly one point of each colour.
	\end{enumerate}
\label{blowitup}
\end{lemma}

\begin{proof}
	We first consider the
	case $u=4$.
		Let the groups be of the form $G_j = \{i_j\mid i\in \mathbb{Z}_4\}$, where $j\in \mathbb{Z}_4$.
	For each $i, j\in \mathbb{Z}_4$, the point $i_j$ receives colour $i$.
	The blocks are listed below so that the ordered tuple  $(x,y,z,w)$  corresponds	to the block
	$\{x_0,y_1,z_2,w_3\}$:
	$$
	\begin{array}{l}
		(0,0,1,2), (0,1,3,1), (0,2,2,0), (0,3,0,3),  \\
		(1,0,0,1), (1,1,2,2), (1,2,3,3), (1,3,1,0),  \\
		(2,0,2,3), (2,1,0,0), (2,2,1,1), (2,3,3,2),  \\
		(3,0,3,0), (3,1,1,3), (3,2,0,2), (3,3,2,1).  \\
	\end{array}$$

Otherwise $u>4$.
	Since $u\equiv 1$ or $4 \pmod{12}$, there exists a quadruple system $Q(u)=(V,{\mathcal B})$ of order $u$.  For each point $v\in V$, construct a group $G(v)=\{(v,i)\mid i\in \mathbb{Z}_4\}$, where $(v,i)$ receives colour $i$ for each $i\in \mathbb{Z}_4$.
	Then, for each block $B\in {\mathcal B}$, place the coloured GDD$(4^4$) constructed
	for the case $u=4$ on the groups $\{G(v)\mid v\in B\}$.
\end{proof}

The following lemma follows from the standard filling-in construction for frames; see Construction~5.1 of~\cite{Stinson87}.  
This construction allows us to build a {\rm KTS}$(4gu+1)$ from a $3$-frame of type $(4g)^u$ by adding a new point $\infty$ to the points of the frame, and for each group $G$, adding the blocks of a {\rm KTS}$(4g+1)$ on the points of $G \cup \{\infty\}$.  Supposing the frame admits a $4$-colouring in which each group contains an equal number of points of each colour, by choosing an equitably $4$-colourable {\rm KTS}$(4g+1)$, we may identify the points appropriately so that the resulting {\rm KTS}$(4gu+1)$ is itself $4$-colourable; this colouring is necessarily equitable since the colour on $\infty$ occurs on $gu+1$ vertices and the other three occur on $gu$ vertices.  

\begin{lemma}
Suppose that there exists a 
$4$-coloured $3$-frame of type $(4g)^u$ such that the points in each group are coloured equitably. Further, suppose that there exists a $4$-chromatic {\rm KTS}$(4g+1)$ that admits an equitable 
$4$-colouring.
	Then there exists a $4$-chromatic {\rm KTS}$(4gu+1)$  that admits an equitable 
	$4$-colouring.
\label{plug4}
\end{lemma}

In Section~\ref{smallorders}, we previously observed the existence of $4$-chromatic Kirkman triple systems of order $33$ which admit a colouring of type $8^39^1$, and hence we have the following.

\begin{lemma}
There exists a $4$-chromatic {\rm KTS}$(33)$ such that the points are equitably coloured for some 
$4$-colouring.
\label{4chrom33}
\end{lemma}

Combining the results of this section, we obtain the following theorem.

\begin{theorem}
	\label{main4colour}
Let $u\equiv 1$ or  $4\pmod{12}$. Then there exists
 a {\rm KTS}$(32u+1)$ with chromatic number $4$.
 Furthermore there is a
$4$-colouring which is
 equitable.
\end{theorem}	
	
\begin{proof}
By Lemma~\ref{4chrom33}, the theorem is true for $u=1$. Otherwise $u\geq 4$ and by
 Lemma \ref{blowitup} there is a $4$-coloured
	{\rm GDD}
	$4^u$ with block size $4$ such that: (a) in any block there are at most two points of the same colour; and (b)
	in any group there is exactly one point of each colour.
	In turn, by Lemma \ref{blowinup} with $g=4$, there
	exists a 
	$4$-coloured $3$-frame
	of type $32^u$ such that each group is equitably coloured. Finally, apply  Lemma \ref{plug4} with $g=8$. Note that the $4$-chromatic equitably coloured KTS$(33)$ needed in this construction is given by Lemma~\ref{4chrom33}.
\end{proof}

\section{Kirkman triple systems with an arbitrary chromatic number}

\label{arbitrary}

In this section our main aim is to prove Theorem~\ref{Thm-main2}. That is, we establish that for each integer $\delta \geq 3$,
there exist infinitely many Kirkman triple systems with chromatic number $\delta$.
We will make use of the following results from the literature.
Note that a {\em partial quadruple system} $Q'$ of order $u$
is a set $U$ of $u$ points and a set ${\mathcal B}'$ of blocks of size $4$ from $U$ such that every pair from $U$ is contained in {\em at most} one block.
If $U\subseteq V$ and ${\mathcal B}'\subseteq {\mathcal B}$ for some
quadruple system $Q=(V,{\mathcal B})$, we say that $Q'$ {\em embeds} in $Q$.

\begin{theorem}[\cite{Ganter,Quackenbush1975}]
	Every partial quadruple  system of order $v$ 
	can be embedded in a quadruple system of order $w$ 
	for some $w\geq v$.
\label{quad}
\end{theorem}

\begin{theorem}[\cite{RS,WZ1,WZ2}]
\label{Rees-Stinson}
Let $w<v$ be positive integers with $w, v \equiv 1$ or $4\pmod{12}$.  Any quadruple system of order $w$ can be embedded in a quadruple system of order $v$ if and only if $v \geq 3w+1$.
\end{theorem}

By combining these two results, we obtain the following corollary.

\begin{corollary}
\label{QuadEmbed}
For any partial quadruple system $Q'(u)$ on $u$ points, there exists an integer $v_0$ such that for all $v \geq v_0$ with $v \equiv 1$ or $4\pmod{12}$, $Q'(u)$ can be embedded in a quadruple system of order $v$.
\end{corollary}

We will also need the following lemma.
\begin{lemma}
Let $\delta\geq 4$ and let $c_0,c_1,c_2\in \mathbb{Z}_{\delta}$.
Then there exists a 
$\delta$-coloured (not necessarily $\delta$-chromatic) resolvable $3$-{\rm GDD} of type $4^3$
with groups $G_j= \{i_j\mid i\in \{0,1,2,3\}\}$, $j\in \{0,1,2\}$,
such that vertex $i_j$ receives colour $c_j+i\pmod{\delta}$ for each $i\in \{0,1,2,3\}$ and $j\in \{0,1,2\}$.
Moreover, if it is not the case that $c_0=c_1=c_2$, then $\{0_0,0_1,0_2\}$ is a block.
\label{Lemma-type43}
\end{lemma}

\begin{proof}
Note that without loss of generality, by cycling colours modulo $\delta$ if needed, we may assume $c_0=0$.
We first consider the case $c_0=c_1=c_2=0$.
Consider the following resolvable $3$-GDD of type $4^3$,  obtained by deleting a group from the construction given in Lemma~\ref{blowitup}.  Specifically, the blocks are as follows, where $(x,y,z)$ denotes the block $\{x_0, y_1, z_2\}$:
$$
\begin{array}{l}
\{(0,0,1), (1,1,2), (2,3,3), (3,2,0)\},\quad
\{(0,1,3), (1,0,0), (2,2,1), (3,3,2)\}, \\
\{(0,2,2), (1,3,1), (2,1,0), (3,0,3)\}, \quad
\{(0,3,0), (1,2,3), (2,0,2), (3,1,1)\}.
\end{array}$$
If point $i_j$ has colour $i$, then all the conditions of the lemma are satisfied.

Otherwise assume $c_0=0$ and $(c_1,c_2)\neq (0,0)$.
We may assume that $c_1\leq c_2$; (otherwise swap groups $G_1$ and $G_2$).
We first consider when $\delta=4$.

We present a list of resolvable $3$-GDDs, each of type $4^3$,  exhaustively covering  every possible case.
We use the same notation as above to describe points, namely, $(x,y,z)$ denotes the block $\{x_0, y_1, z_2\}$.
Firstly, the case $(c_1,c_2)\in \{(0,1),(0,2),(0,3),
(1,1),(2,2),(3,3)\}$:
$$
\begin{array}{l}
\{(0,0,0), (1,1,1), (2,2,2), (3,3,3)\},\quad
\{(0,1,2), (1,0,3), (2,3,0), (3,2,1)\}, \\
\{(0,2,3), (1,3,2), (2,0,1), (3,1,0)\}, \quad
\{(0,3,1), (1,2,0), (2,1,3), (3,0,2)\}.
\end{array}$$
Next,
the case $(c_1,c_2)=(1,2)$:
$$
\begin{array}{l}
\{(0,0,0), (1,2,1), (2,3,2), (3,1,3)\},\quad
\{(0,1,2), (1,3,3), (2,2,0), (3,0,1)\}, \\
\{(0,2,3), (1,0,2), (2,1,1), (3,3,0)\}, \quad
\{(0,3,1), (1,1,0), (2,0,3), (3,2,2)\}.
\end{array}$$
Now, the case $(c_1,c_2)=(1,3)$:
$$
\begin{array}{l}
\{(0,0,0), (1,2,2), (2,3,3), (3,1,1)\},\quad
\{(0,2,1), (1,0,3), (2,1,2), (3,3,0)\}, \\
\{(0,1,3), (1,3,1), (2,2,0), (3,0,2)\}, \quad
\{(0,3,2), (1,1,0), (2,0,1), (3,2,3)\}.
\end{array}$$
And finally, a solution for $(c_1,c_2)=(2,3)$ is obtained by swapping the second and third coordinates in each of the triples in the case $(c_1,c_2)=(1,3)$.

Otherwise $c_0=0$, $(c_1,c_2)\neq (0,0)$ and $\delta>4$.
Let $X_0=\{0,1,2,3\}$, $X_1=\{c_1,c_1+1,c_1+2,c_1+3\}$ and  $X_2=\{c_2,c_2+1,c_2+2,c_2+3\}$, where elements in these sets are evaluated in $\mathbb{Z}_{\delta}$.
That is, $X_0$, $X_1$ and $X_2$ are the sets of colours used in groups $G_0$, $G_1$ and $G_2$, respectively.
Considering $X_0$, monochromatic blocks can only occur in colours $0$, $1$, $2$ or $3$.
Thus if either
$X_1$ or $X_2$ is disjoint from $X_0$, we have a solution using any of the above GDDs from the case $\delta=4$ and $(c_1,c_2)\neq (0,0)$.

It remains to consider the case that $|X_1\cap X_0|\geq 1$
and  $|X_2\cap X_0|\geq 1$.
	For $j\in \{1,2\}$,
	if $c_j\in X_0$, set $c_j'=c_j$; otherwise
	$0<\delta-3\leq c_j\leq \delta-1$, and define $c_j'=4-\delta+c_j\in\{1,2,3\}$.
 Since $(c_1',c_2')\neq (0,0)$,
from above there exists a resolvable $3$-GDD of type $4^3$ which satisfies the conditions of the lemma for
$\delta=4$, where $c_1$ and $c_2$ are replaced by $c_1'$ and $c_2'$, respectively.

In the following paragraphs, we recolour (and possibly relabel) points in groups $G_1$ and $G_2$
  to obtain
a 
$\delta$-colourable  resolvable $3$-GDD of type $4^3$ which satisfies the conditions of the lemma.

	For each $j \in \{1,2\}$, we consider three cases.
	If $c_j\in X_0\setminus \{\delta-3,\delta-2,\delta-1\}$, then for each integer $i$ such that $4-c_j\leq i\leq 3$, recolour the point $i_j$ with colour $i+c_j$.
	If $c_j\in  \{\delta-3,\delta-2,\delta-1\}\setminus X_0$, then for each integer $i$ such that $0\leq i\leq \delta-c_i$ recolour the points $i_j$ with colour $c_j+i$.
	Lastly, if
	$c_j\in  \{\delta-3,\delta-2,\delta-1\}\cap X_0$, then
	necessarily
	$(\delta,c_j)\in \{(5,2),(5,3),(6,3)
	\}$, which we consider individually as subcases below.

	If $\delta=5$ and
	$c_j=2$, points
	$0_j$, $1_j$,
	$2_j$ and $3_j$ are coloured with colours $2$, $3$, $0$ and $1$, respectively.
	Recolour vertex $3_j$ with colour $4$ then swap vertices $2_j$ and $3_j$ (in all blocks), keeping the new colours.

	If $\delta=5$ and
	$c_j=3$, points
	$0_j$, $1_j$,
	$2_j$ and $3_j$ are coloured with colours $3$, $0$, $1$ and $2$, respectively.
	Recolour vertex $3_j$ with colour $4$ then replace vertex $1_j$ with $3_j$, vertex $3_j$ with $2_j$ and vertex $2_j$ with $1_j$ (in all blocks), keeping the new colours.

	Finally, if $\delta=6$ and
	$c_j=3$, points
	$0_j$, $1_j$,
	$2_j$ and $3_j$ are coloured with colours $3$, $0$, $1$ and $2$, respectively.
	Recolour vertex $2_j$ with colour $5$, recolour vertex $3_j$ with colour $4_j$, then swap vertex $1_j$ with $3_j$ (in all blocks), keeping the new colours.
\end{proof}

\begin{example}
For example, in the previous theorem,   suppose $\delta=5$,
$c_0=0$, $c_1=2$ and $c_2=4$. Then   $X_1=\{2,3,4,0\}$ and
$X_2=\{4,0,1,2\}$.   In turn, $c_1'=
2$ and $c_2'=3$.

 Then, using
 notation from the theorem, the blocks are as follows:
 $$
 \begin{array}{l}
 \{(0,0,0), (1,3,2), (2,2,3), (3,1,1)\},\quad
 \{(0,1,2), (1,2,0), (2,3,1), (3,0,3)\}, \\
 \{(0,2,1), (1,1,3), (2,0,2), (3,3,0)\}, \quad
 \{(0,3,3), (1,0,1), (2,1,0), (3,2,2)\}.
 \end{array}$$
\end{example}

\begin{lemma}
For each $\delta\geq 4$,
and list of (possibly non-distinct) colours $c_0,c_1,c_2,c_3\in \mathbb{Z}_{\delta}$,
there exists a 
$\delta$-coloured $3$-frame of type $8^4$ with groups
$G_j= \{i_j\mid i\in \{0,1,\ldots,7\}\}$,   $j\in \{0,1,2,3\}$, such that:
\begin{enumerate}
\item[a)] vertex $i_j$ has colour $c_j+i\pmod{\delta}$ for each $j\in \{0,1,2,3\}$;
\item[b)] there exists a block
 $\{0_0,0_1,0_2\}$ in the frame, unless $c_0=c_1=c_2$.
\end{enumerate}
\label{Lemma-type83}
\end{lemma}

\begin{proof}
By Theorem~\ref{framebyanyothername}, there exists a $3$-frame of type $2^4$ given by $(V,{\mathcal G},{\mathcal B})$. Here \linebreak
 $V=\{x_j\mid x\in \{0,1\}, j\in \{0,1,2,3\}\}$,  ${\mathcal G}=\{\{x_j\mid x\in \{0,1\}\}, j\in\{0,1,2,3\}\}$ and without loss of generality, we assume that $(0_0, 0_1,0_2)\in {\mathcal B}$.

Next, we apply the {``Inflation by TD''} Construction (Construction 3.2 from~\cite{Stinson87}), giving each point in
the above $3$-frame weight $4$ and thus placing the resolvable $3$-GDD of type $4^3$
from Lemma~\ref{Lemma-type43}
on each blown up block, to obtain a $3$-frame of type $8^4$.
 Here for each $x\in\{0,1\}$, and $j\in \{0,1,2,3\}$, $x_j\in V$ is replaced by the set of points $\{(4x+y)_j, y\in \{0,1,2,3\}\}$. In addition, point $(4x+y)_j$, receives colour $c_j+4x+y \pmod \delta$, so that (a) is satisfied.

 	We note that this colouring is consistent with that of each of the
 $3$-GDD of type $4^3$ from Lemma~\ref{Lemma-type43} placed on the blocks, using $c_j+4x\pmod \delta$ in place of $c_j$ in that lemma.
  	Finally, note that since $(0_0, 0_1,0_2)$ was a block in the original frame and the ingredient 3-GDDs came from Lemma~\ref{Lemma-type43}, (b) holds.
\end{proof}

\begin{theorem}
\label{Thm-arbit}
Let $\delta\geq 4$. There exist infinitely many Kirkman triple systems with chromatic number $\delta$.
\end{theorem}

\begin{proof}
Recall that, as mentioned in the introduction,
 there are infinitely many Steiner triple systems with chromatic number $\delta$~(see \cite{dBPR}).		
Let $S=(V,{\mathcal B})$ be such a Steiner triple system on $|V|=v$ points with a 
$\delta$-colouring.
Let ${\mathcal B}=\{B_1,B_2,\dots ,B_b\}$,
where $b=|{\mathcal B}|$.
Let $X=\{x_1,x_2,\dots x_b\}$ be a set of  distinct points, where $X$ is disjoint from $V$.

 Let $Q'(v+b)$ be the partial quadruple system of order $v+b$ defined by adding the point $x_i$ to block $B_i$ for each $i\in \{1,2,\ldots,b\}$.
 By Corollary~\ref{QuadEmbed},
there exist infinitely many integers $w$ such that
$Q'(v+b)$ embeds in a quadruple system $Q(w)$ of order $w$.
 Colour the points of $Q(w)\setminus V$ arbitrarily with $\delta$ colours from the set $\{0, \ldots, \delta-1\}$.

 Next, for each point $j$ in $Q(w)$, let $\{i_j \mid
 i \in \{0,1,\ldots,7\}\}$ be a group of points such that if $j$ has colour $c$, then $i_j$ has colour $c+i\pmod{\delta}$.
 For each block $B$ in $Q(w)$, place a
 $3$-frame of type $8^4$ on the
 point set
 $\{i_j\mid j\in B, i\in \{0,1,\ldots,7\}\}$ in such a way that (a) no triple is monochromatic; and (b)
 if $B$ contains an original triple $\{x,y,z\}$ from $S$, the triple $\{0_x,0_y,0_z\}$ is included in the corresponding $3$-frame.
 This is possible by Lemma~\ref{Lemma-type83}.

Next, add an infinity point which can receive any colour. Add a copy of KTS$(9)$ to each group together with the infinity point, so that no block is monochromatic. This is straightforward to do when $\delta\geq 4$.
The resultant 
triple system $S'$ is a KTS$(8w+1)$~\cite[Construction 5.1]{Stinson87}, which is
$\delta$-coloured by construction; moreover it contains $S$ as a subsystem so has chromatic number $\delta$.
\end{proof}

Theorem~\ref{Thm-main2}, that for each $\delta\geq 3$, there are infinitely many integers $v$ such that there exists a {\rm KTS$(v)$} with chromatic number $\delta$, now follows from
		Theorems~\ref{Thm-main1}
 and~\ref{Thm-arbit}.

\section{Kirkman triple systems from quadruple systems}

\label{quadruple}

We first review a standard construction for Kirkman triple systems that makes use of quadruple systems.
Given a quadruple system $Q$ of order $v$
on vertex set  $\{q_0,q_1,q_2,\dots ,q_{v-1}\}$,
we build a Kirkman triple system $K(Q)$ of order $2v+1$ as follows.

The vertex set $V$ of $K(Q)$ is given by:
$$\{\infty\}\cup
\{q_i,q_i'\mid i\in \{0,1,\ldots,v-1\}\}.$$
For each block $\{w,x,y,z\}$ from
$Q$, place
a KTS$(9)$ on
the vertex set
$$\{\infty\}\cup
\{w,x,y,z,w',x',y',z'\}$$
in such a way that
the triples including $\infty$ are
$\{\infty,w,w'\}$,
$\{\infty,x,x'\}$,
$\{\infty,y,y'\}$ and
$\{\infty,z,z'\}$.
Observe that such a KTS$(9)$ placement is not unique, so in turn $K(Q)$ may not be unique.

It is well known that
$K(Q)$ is a Kirkman triple system; see
Theorem~5.1.3 from \cite{LindnerRodger}.
Building on this construction, we obtain some results which may be useful in future work to construct Kirkman triple systems with fixed chromatic number.

\begin{theorem}\label{ThmQKQcolouring}
	If $Q$ is a quadruple system of order $v$ and if there is a $K(Q)$ which is 
	$\delta$-colourable,
 then $Q$ has a 
$\delta$-colouring.	
\end{theorem}

\begin{proof}
	Let $A=\{q_0,q_1,\dots ,q_{v-1}\}$ and $B=\{q_0',q_1',\dots ,q_{v-1}'\}$.
	From the definition of $K(Q)$, we may assume that the  vertices of $K(Q)$ are labelled with
	$\{\infty\}\cup A \cup B$
	in such a way that for each
	$\{w,x,y,z\}\subset A$
	that is a block of $Q$,
	there is a KTS$(9)$ on the vertex set:
	$$\{\infty,w,x,y,z,w',x',y',z'\}.$$
	Moreover, $\{\infty,q_i,q_i'\}$ is a triple for each $0\leq i \leq v-1$.
	
	Fix a 
	$\delta$-colouring of the points of $K(Q)$.
	Observe that in this colouring, when restricted to the point set $A$ of $Q$, the block $\{w,x,y,z\}$ is not monochromatic.
	We therefore conclude that the 
	$\delta$-colouring yields
	a $\delta$-colouring of $Q$ such that no block is monochromatic.
\end{proof}

\begin{corollary} \label{Cor:KQlowerbound}
	If $Q$ is a quadruple system, then 
	for any $K(Q)$,
	$\chi(K(Q))\geq \chi(Q)$.
\end{corollary}

\begin{theorem} \label{Thm2QKQcolouring}
	Let $\delta \geq 2$.
	If 
	$Q$ is
	 a 
	$\delta$-colourable quadruple system 
 	of order $v$, then there 
	is a 
	$K(Q)$ which is a  
	$2\delta$-colourable Kirkman triple system of order $2v+1$.
\end{theorem}

\begin{proof}
	Let $Q$ be a 
	$\delta$-colourable quadruple system of order $v$.
	Let $A=\{q_0,q_1,\dots ,q_{v-1}\}$ and
	$B=\{q_0',q_1',\dots ,q_{v-1}'\}$.
	Fix a 
	$\delta$-colouring of $Q$ on the vertex set $A$ using the colour set $\{1,2,\dots ,\delta\}$.
	
	Consider a Kirkman triple system $K(Q)$ based on $Q$ as described above.
	By the definition of $K(Q)$, 
	 $\{\infty,q_i,q_i'\}$ is a triple for each $0\leq i \leq v-1$, and
		for each $\{w,x,y,z\}\subset A$
	that is a block of $Q$,
	 $K(Q)$ contains the blocks of a KTS$(9)$ on the vertex set
	$$\{\infty,w,x,y,z,w',x',y',z'\}.$$
	Moreover, for each block $\{w,x,y,z\}$, since the block is not monochromatic then we may select $x$, $y$ and $z$ to be points that do not all have the same colour,
	and include
	 the following triples in $K(Q)$:
	 $$\begin{array}{cccc}
	 	\{x,y,z\}, & \{w,y,z'\}, & \{x',w,z\}, & \{x,w,y'\}, \\
	 	\{x',y',z'\}, & \{w',y',z\}, & \{x,w',z'\}, & \{x',w',y\}. \\
	 \end{array}$$
	
	We now colour the points of $B\cup \{\infty\}$ using the extended colour set
	$\{1,2,\dots ,2\delta\}$.
	Whenever vertex $w\in A$ is coloured with $j$, let $w'\in B$ be coloured with $\delta+j$.
	Finally, colour $\infty$ with colour $1$.
	Hence, for each block $\{w,x,y,z\}$ in $Q$, since the points in the set $\{x,y,z\}$ are mapped to at least two different colours,
	the points in the set $\{x',y',z'\}$ are also mapped to at least two different colours.
	
	Observe that for each of the triples
	$$\{\infty,w,w'\},
	\{\infty,x,x'\},
	\{\infty,y,y'\},
	\{\infty,z,z'\},
	\{x,y,z\},
	\{x',y',z'\}$$
	the points are mapped to at least two different colours.
	Each of the remaining six triples of the KTS$(9)$ arising from $\{w,x,y,z\}$ intersects both $A$ and $B$ non-trivially,
	so again receives at least two distinct colours.
	
	In conclusion, the $K(Q)$ constructed is 
	$2\delta$-colourable.
\end{proof}

We note, however, that if $Q$ is $\delta$-chromatic, then $K(Q)$ need not be $2\delta$-chromatic, as the following example demonstrates.
\begin{example} \label{Q13Example}
Consider the quadruple system $Q(13)$ with the following blocks:
\[
\begin{array}{lllllll}
\{0,1,3,9\}, & \{1,2,4,10\}, & \{2,3,5,11\}, & \{3,4,6,12\}, & \{4,5,7,0\}, \\
 \{5,6,8,1\}, &
\{6,7,9,2\}, &
 \{7,8,10,3\}, & \{8,9,11,4\}, & \{9,10,12,5\}, \\ \{10,11,0,6\},
& \{11,12,1,7\}, & \{12,0,2,8\}.
\end{array}
\]
Then $Q(13)$ is $2$-chromatic, for example using the colour class $\{0,1,2,3,4,5,6\}$ and its complement.
Next, we construct a particular $K(Q)$; in an abuse of bracket notation, we assume that the elements of the sets are each ordered as written, 
	replacing each $\{w,x,y,z\}$ 
with 
$$
\begin{array}{cccc}
	\{\infty,w,w'\}, & \{\infty,x,x'\}, & \{\infty,y,y'\}, & \{\infty,z,z'\}, \\
	\{x,y,z\}, & \{w,y,z'\}, & \{x',w,z\}, & \{x,w,y'\}, \\
	\{x',y',z'\}, & \{w',y',z\}, & \{x,w',z'\}, & \{x',w',y\}. \\
\end{array}
$$
(So, for example, for the block $\{12,0,2,8\}$ 
we assume that $w=12$, $x=0$, $y=2$ and $z=8$.)
However, 
this
$K(Q)$ is $3$-chromatic; one $3$-colouring has colour classes:
\[
\{0,1,2,3,4,5,6,0',2'\}, \quad  \{7,8,10,12,6',8',9',11',12'\}, \quad \{9,11,1',3',4',5',7',10',\infty\}.
\]
\end{example}

More generally, we give some conditions on a $\delta$-colourable quadruple system ($\delta \geq 3$) that 
yield a
  Kirkman triple system $K(Q)$ 
which
  is $(\delta+1)$-colourable.

\begin{theorem}\label{Thm:delta+1}
	Suppose there exists a quadruple system $Q$ of order $v$ and a $\delta$-colouring of its points such that each block receives at least $3$ distinct colours.
	Then there exists a Kirkman triple system 
	$K(Q)$ of order $2v+1$ that is 
	$(\delta+1)$-colourable.
\end{theorem}

\begin{proof}
	Let $Q$ be such
	 a quadruple system
	of order $v$.
	Let $A=\{q_0,q_1,\dots ,q_{v-1}\}$ and $B=\{q_0',q_1',\dots ,q_{v-1}'\}$.
	Fix a colouring of $Q$ on the point set $A$, using the colour set $\{1,2,\dots ,\delta\}$.

	For each block $\{w,x,y,z\}$ of $Q$, assuming without loss of generality that $x$, $y$ and $z$ have different colours, 
	add the following triples to form a KTS$(9)$: 
$$
\begin{array}{cccc}
	\{\infty,w,w'\}, & \{\infty,x,x'\}, & \{\infty,y,y'\}, & \{\infty,z,z'\}, \\
	\{x,y,z\}, & \{w,y,z'\}, & \{x',w,z\}, & \{x,w,y'\}, \\
	\{x',y',z'\}, & \{w',y',z\}, & \{x,w',z'\}, & \{x',w',y\}. \\
\end{array}
$$

	Now we label the points of $B \cup \{\infty\}$ with the colour set $\{1,2,\dots ,\delta,\delta+1\}$.
	Label $\infty$ with colour $\delta+1$.
	Next, whenever vertex $w\in A$ is coloured with $j$, let $w'\in B$ be also coloured with $j$.
	It is now easily observed that each of the triples of the KTS$(9)$ arising from each block $\{w,x,y,z\}$ has at least two colours.
\end{proof}

We finish this section by using known existence results for quadruple systems along with Theorem~\ref{Thm2QKQcolouring} to construct Kirkman triple systems with small chromatic number.

\begin{theorem} {\rm \cite{FGLR2002,HLP1990}}
	There is a 
	$2$-colourable quadruple system of order $3n+1$  if and only if
	$n\equiv 0$ or $1\pmod{4}$.
\end{theorem}	

\begin{corollary}
	For each $n\equiv 0$ or $1\pmod{4}$,
	there exists a Kirkman triple system
$K(6n+3)$, say $S$, such that $\chi(S)\leq 4$.
	\label{Cor-4fromquad}
\end{corollary}

\section{Conclusion}
\label{conclusion}

Theorem~\ref{Thm-main2} posits, for each integer $\delta \geq 3$, the existence of infinitely many $\delta$-chromatic Kirkman triple systems.  In the particular case that $\delta=3$, Theorem~\ref{Thm-main1} shows that there is a $3$-chromatic Kirkman triple system for every order $v \equiv 3 \pmod{6}$,
thereby completely determining the spectrum for $3$-chromatic Kirkman triple systems.

For $\delta=4$, Theorem~\ref{main4colour} gives a $4$-chromatic Kirkman triple system whenever $v \equiv 33$ or $129 \pmod{384}$.  It is known that a $4$-chromatic Steiner triple system exists for every admissible order $v \geq 21$~\cite{CFGGKOPP, Haddad}; whether the same is true for Kirkman triple systems remains open.  We tested the 66,937 KTS(21) with nontrivial automorphism group for colouring properties, and found that they are all $3$-chromatic; thus, a $4$-chromatic KTS(21) would necessarily be free of nontrivial automorphisms.
\begin{question} \label{Question-kts21}
Does there exist a $4$-chromatic ${\rm KTS}(21)$?
\end{question}
Given the complete existence result for $4$-chromatic KTS($v$) for orders in several congruence classes, it seems likely that they will exist more generally.  We thus make the following conjecture.
\begin{conjecture} \label{Conj-4chrom}
There is a $4$-chromatic Kirkman triple system of every order $v \equiv 3 \pmod{6}$ with $v \geq v_0$ for some integer $v_0 \geq 21$.
\end{conjecture}
The value of $v_0$ in Conjecture~\ref{Conj-4chrom} is open for debate, and clearly the actual value (if it exists) is dependent on the answer to Question~\ref{Question-kts21}.  The smallest order for which we have actually found a $4$-chromatic Kirkman triple system is $v=33$, and so we conjecture that $v_0 \leq 33$.

Considering higher chromatic numbers, de~Brandes, Phelps and R\"{o}dl~\cite{dBPR} showed that for every integer $\delta \geq 3$, there exists a value $v_0$ for which there is a $\delta$-chromatic Steiner triple system of every admissible order $v \geq v_0$.  We conjecture that an analogous result holds for Kirkman triple systems.
\begin{conjecture}
For every integer $\delta \geq 3$, there exists an integer $v_0$ such that whenever $v \geq v_0$ and $v \equiv 3 \pmod{6}$, there is a $\delta$-chromatic Kirkman triple system of order $v$.
\end{conjecture}

In Section~\ref{quadruple}, we explored colouring properties of Kirkman triple systems formed from quadruple systems.  Specifically, Theorems~\ref{ThmQKQcolouring} and~\ref{Thm2QKQcolouring} together imply that if $Q$ is a $\delta$-chromatic quadruple system, then
there exists a $K(Q)$ such that 
 $\delta \leq \chi(K(Q)) \leq 2\delta$.  Applying this result with $2$-chromatic quadruple systems, we obtain Kirkman triple systems of orders congruent to $3$ and $9\pmod{24}$ whose chromatic number is either 3 or 4 (Corollary~\ref{Cor-4fromquad}).  While Theorem~\ref{Thm:delta+1} gives conditions on the quadruple system $Q$ which would imply that $\chi(K(Q)) \leq \delta+1$
for some $K(Q)$
 , we note that these conditions cannot be satisfied for a $2$-colouring of $Q$. Nevertheless, Example~\ref{Q13Example} shows that a $2$-chromatic quadruple system may give rise to a $3$-chromatic Kirkman triple system.
We therefore ask the following question.
\begin{question}
	Does there exist a  quadruple system $Q$ and a $K(Q)$ where 
	$Q$ is $2$-chromatic and 
	$K(Q)$ is $4$-chromatic?
\end{question}
\noindent
More generally, are there conditions on a $2$-chromatic quadruple system $Q$ that characterize whether $K(Q)$ is $3$- or $4$-chromatic?

For a $\delta$-chromatic quadruple system with $\delta \geq 3$, we ask if the lower bound from Corollary~\ref{Cor:KQlowerbound} on the chromatic number of $K(Q)$ is tight; a positive answer would yield an additional method to construct Kirkman triple systems with specified chromatic number.
\begin{question}
Let $\delta \geq 3$ be an integer.  Does there exist a $\delta$-chromatic quadruple system $Q$ 
and a  $K(Q)$
such that $\chi(K(Q))=\delta$?
\end{question}

We also ask whether the upper bound from Theorem~\ref{Thm2QKQcolouring} is tight, particularly if $\delta \geq 3$.
\begin{question}
Does there exist a $\delta$-chromatic quadruple system $Q$ 
and a  $K(Q)$
such that $\chi(K(Q))=2\delta$?
\end{question}

Finally, we note that the question of colouring Kirkman triple systems can be generalized in several ways.  It would be interesting to study colourings of resolvable BIBD$(v,k,\lambda)$ for arbitrary values of $k$ and/or $\lambda$.  Another natural extension would be to consider colourings of resolvable $\ell$-cycle decompositions of the complete graph (or the complete graph with a 1-factor $I$ removed, for even orders).  More generally, the {\em Oberwolfach Problem} OP$(F)$ asks, given a 2-regular graph $F$ of order $v$, whether there exists a 2-factorization of $K_v$ or $K_v-I$ in which each 2-factor is isomorphic to $F$ (see~\cite{survey} for further details); requiring a solution to the Oberwolfach Problem or its variants (such as the Hamilton-Waterloo Problem) which satisfies vertex-colouring properties provides an open area of study.

\section{Acknowledgements}

We thank Petteri Kaski for providing us with the collection of 66,937 KTS(21) having nontrivial automorphism groups.  These were determined in conjunction with the classification of the STS(21)s admitting a nontrivial group of automorphisms~\cite{Kaski2005}.

Authors Burgess, Danziger and Pike acknowledge research grant support from NSERC Discovery Grants RGPIN-2019-04328,  RGPIN-2022-03816 and RGPIN-2022-03829, respectively.

\section*{Appendix}

In Section~\ref{smallorders} we remarked that each of the thirty KTS(33) presented by Tonchev and Vanstone in~\cite{TV}
is 4-chromatic, and moreover, that each has a 4-colouring of type $8^3 9^1$.  
Below we provide details for how to 
obtain such a colouring for each of these designs.  
We also found that 13 of the KTS(33) which are 1-rotational over the cyclic group from~\cite{BurattiZuanni} are 3-chromatic, and provide the details of a colouring of one such system.

We first consider the systems from~\cite{TV}. In each case the point set is $V=\mathbb{Z}_{33}$.
For the first 29 systems, we provide the blocks of a starter parallel class $\cP$ from which ten other parallel classes
can be obtained by adding $t \in \{3,6,\ldots,30\}$ to each point of each block of $\cP$.
We also provide a set $\cS$ of five starter blocks; for each $B \in \cS$,
$\big\{ B+t \mid t \in \{0,3,6,\ldots,30\} \big\}$ is a parallel class.
For the 30th system, we provide the blocks of a starter parallel class $\cP$ from which fifteen other parallel classes
can be obtained by repeatedly applying the permutation $(0,1,2,\ldots,31)(32)$ to each point of each block of $\cP$.

For each system we indicate how to achieve a 4-colouring of type $8^3 9^1$ by showing a sequence of length 33 in which the colour assigned to point 0 is followed by the colour of point 1, etc.

\hspace*{19mm}
\begin{enumerate}[label=\textbf{KTS(33) \#\arabic*:~},leftmargin=32mm]

\item
\begin{tabbing}
$\cP=\big\{$\=
$\{ 1, 2, 4\},\{ 5, 6, 8\},\{ 18, 19, 21\},\{ 10, 16, 28\},\{ 11, 26, 32\},\{ 9, 24, 30\}$,
\\ \> $\{ 7, 14, 23\},\{ 3, 12, 29\},\{ 15, 22, 31\},\{ 13, 17, 27\},\{ 20, 25, 0\}
\big\}$
\end{tabbing}

$\cS=\big\{ \{2, 6, 16\},\{ 3, 7, 17\},\{ 1, 6, 14\},\{ 3, 8, 16\},\{ 1, 12, 23\} \big\}$

4-colouring: {\tt 111221112211122233324433344433444}

\item
\begin{tabbing}
$\cP=\big\{$\=
$\{ 1, 2, 4\},\{ 8, 9, 11\},\{ 18, 19, 21\},\{ 13, 25, 31\},\{ 14, 26, 32\},\{ 12, 24, 30\}$,
\\ \> $\{ 5, 22, 29\},\{ 3, 20, 27\},\{ 7, 16, 0\},\{ 10, 15, 23\},\{ 6, 17, 28\}
\big\}$
\end{tabbing}

$\cS=\big\{ \{1, 5, 15\},\{ 2, 6, 16\},\{ 3, 7, 17\},\{ 2, 7, 15\},\{ 3, 8, 16\} \big\}$

4-colouring: {\tt 111221112211122233324433344433444}

\item
\begin{tabbing}
$\cP=\big\{$\=
$\{ 1, 2, 4\},\{ 14, 15, 17\},\{ 30, 31, 0\},\{ 10, 22, 28\},\{ 5, 11, 26\},\{ 6, 12, 27\}$,
\\ \> $\{ 7, 16, 23\},\{ 3, 20, 29\},\{ 9, 18, 25\},\{ 19, 24, 32\},\{ 8, 13, 21\}
\big\}$
\end{tabbing}

$\cS=\big\{ \{1, 5, 15\},\{ 2, 6, 16\},\{ 3, 7, 17\},\{ 3, 8, 16\},\{ 1, 12, 23\} \big\}$

4-colouring: {\tt 111221112211122233324343444334443}

\item
\begin{tabbing}
$\cP=\big\{$\=
$\{ 1, 2, 4\},\{ 8, 9, 11\},\{ 24, 25, 27\},\{ 10, 16, 28\},\{ 5, 17, 32\},\{ 12, 18, 30\}$,
\\ \> $\{ 7, 14, 31\},\{ 3, 20, 29\},\{ 6, 15, 22\},\{ 19, 23, 0\},\{ 13, 21, 26\}
\big\}$
\end{tabbing}

$\cS=\big\{ \{2, 6, 16\},\{ 3, 7, 17\},\{ 1, 6, 26\},\{ 2, 7, 27\},\{ 1, 12, 23\} \big\}$

4-colouring: {\tt 111221112211122233324343434343444}

\item
\begin{tabbing}
$\cP=\big\{$\=
$\{ 1, 2, 4\},\{ 20, 21, 23\},\{ 12, 13, 15\},\{ 10, 22, 28\},\{ 5, 11, 26\},\{ 3, 18, 30\}$,
\\ \> $\{ 7, 14, 31\},\{ 8, 17, 24\},\{ 9, 16, 0\},\{ 6, 25, 29\},\{ 19, 27, 32\}
\big\}$
\end{tabbing}

$\cS=\big\{ \{2, 6, 16\},\{ 3, 7, 17\},\{ 1, 6, 26\},\{ 2, 7, 27\},\{ 1, 12, 23\} \big\}$

4-colouring: {\tt 111221112211122233324343434343444}

\item
\begin{tabbing}
$\cP=\big\{$\=
$\{ 1, 2, 4\},\{ 17, 18, 20\},\{ 24, 25, 27\},\{ 10, 16, 28\},\{ 5, 11, 23\},\{ 3, 9, 21\}$,
\\ \> $\{ 7, 14, 31\},\{ 8, 15, 32\},\{ 6, 13, 30\},\{ 22, 26, 12\},\{ 29, 0, 19\}
\big\}$
\end{tabbing}

$\cS=\big\{ \{3, 7, 26\},\{ 1, 6, 14\},\{ 2, 7, 15\},\{ 3, 8, 16\},\{ 1, 12, 23\} \big\}$

4-colouring: {\tt 111221112211122233324433344433444}

\item
\begin{tabbing}
$\cP=\big\{$\=
$\{ 1, 2, 4\},\{ 8, 9, 11\},\{ 18, 19, 21\},\{ 13, 25, 31\},\{ 14, 26, 32\},\{ 12, 24, 30\}$,
\\ \> $\{ 5, 22, 29\},\{ 3, 20, 27\},\{ 7, 16, 0\},\{ 10, 15, 23\},\{ 28, 6, 17\}
\big\}$
\end{tabbing}

$\cS=\big\{ \{1, 5, 24\},\{ 2, 6, 25\},\{ 3, 7, 26\},\{ 2, 7, 15\},\{ 3, 8, 16\} \big\}$

4-colouring: {\tt 111221112211122233324433344433444}

\item
\begin{tabbing}
$\cP=\big\{$\=
$\{ 1, 2, 4\},\{ 14, 15, 17\},\{ 30, 31, 0\},\{ 22, 28, 10\},\{ 5, 11, 26\},\{ 6, 12, 27\}$,
\\ \> $\{ 16, 23, 7\},\{ 29, 3, 20\},\{ 18, 25, 9\},\{ 19, 24, 32\},\{ 8, 13, 21\}
\big\}$
\end{tabbing}

$\cS=\big\{ \{1, 5, 24\},\{ 2, 6, 25\},\{ 3, 7, 26\},\{ 3, 8, 16\},\{ 1, 12, 23\} \big\}$

4-colouring: {\tt 111221112211122233324433344433444}

\item
\begin{tabbing}
$\cP=\big\{$\=
$\{ 1, 2, 4\},\{ 17, 18, 20\},\{ 24, 25, 27\},\{ 10, 16, 28\},\{ 5, 11, 23\},\{ 3, 9, 21\}$,
\\ \> $\{ 7, 14, 31\},\{ 8, 15, 32\},\{ 6, 13, 30\},\{ 22, 26, 12\},\{ 29, 0, 19\}
\big\}$
\end{tabbing}

$\cS=\big\{ \{3, 7, 26\},\{ 1, 6, 26\},\{ 2, 7, 27\},\{ 3, 8, 28\},\{ 1, 12, 23\} \big\}$

4-colouring: {\tt 111221112211122233324343434343444}

\item
\begin{tabbing}
$\cP=\big\{$\=
$\{ 1, 2, 5\},\{ 8, 9, 12\},\{ 24, 25, 28\},\{ 4, 6, 18\},\{ 17, 19, 31\},\{ 11, 30, 32\}$,
\\ \> $\{ 7, 13, 22\},\{ 14, 20, 29\},\{ 3, 21, 27\},\{ 10, 15, 23\},\{ 16, 26, 0\}
\big\}$
\end{tabbing}

$\cS=\big\{ \{2, 7, 15\},\{ 3, 8, 16\},\{ 1, 8, 24\},\{ 3, 10, 26\},\{ 1, 12, 23\} \big\}$

4-colouring: {\tt 111122211112222333314423343434444}

\item
\begin{tabbing}
$\cP=\big\{$\=
$\{ 1, 2, 5\},\{ 14, 15, 18\},\{ 9, 10, 13\},\{ 12, 31, 0\},\{ 4, 23, 25\},\{ 8, 27, 29\}$,
\\ \> $\{ 7, 16, 22\},\{ 11, 20, 26\},\{ 3, 21, 30\},\{ 19, 24, 32\},\{ 6, 17, 28\}
\big\}$
\end{tabbing}

$\cS=\big\{ \{3, 8, 16\},\{ 2, 7, 15\},\{ 1, 8, 18\},\{ 2, 9, 19\},\{ 3, 10, 20\} \big\}$

4-colouring: {\tt 111122211112222333344431233434444}

\item
\begin{tabbing}
$\cP=\big\{$\=
$\{ 1, 2, 5\},\{ 20, 21, 24\},\{ 9, 10, 13\},\{ 6, 25, 27\},\{ 17, 19, 31\},\{ 11, 30, 32\}$,
\\ \> $\{ 7, 16, 22\},\{ 14, 23, 29\},\{ 3, 12, 18\},\{ 8, 28, 0\},\{ 4, 15, 26\}
\big\}$
\end{tabbing}

$\cS=\big\{ \{3, 8, 16\},\{ 2, 7, 15\},\{ 1, 8, 18\},\{ 2, 9, 19\},\{ 3, 10, 20\} \big\}$

4-colouring: {\tt 111122211112222333344431233434444}

\item
\begin{tabbing}
$\cP=\big\{$\=
$\{ 1, 2, 5\},\{ 14, 15, 18\},\{ 9, 10, 13\},\{ 31, 0, 12\},\{ 23, 25, 4\},\{ 27, 29, 8\}$,
\\ \> $\{ 16, 22, 7\},\{ 20, 26, 11\},\{ 30, 3, 21\},\{ 19, 24, 32\},\{ 28, 6, 17\}
\big\}$
\end{tabbing}

$\cS=\big\{ \{3, 8, 16\},\{ 2, 7, 15\},\{ 1, 8, 24\},\{ 2, 9, 25\},\{ 3, 10, 26\} \big\}$

4-colouring: {\tt 111122211112222333144323343434444}

\item
\begin{tabbing}
$\cP=\big\{$\=
$\{ 1, 2, 5\},\{ 20, 21, 24\},\{ 9, 10, 13\},\{ 25, 27, 6\},\{ 17, 19, 31\},\{ 30, 32, 11\}$,
\\ \> $\{ 16, 22, 7\},\{ 23, 29, 14\},\{ 12, 18, 3\},\{ 28, 0, 8\},\{ 4, 15, 26\}
\big\}$
\end{tabbing}

$\cS=\big\{ \{3, 8, 16\},\{ 2, 7, 15\},\{ 1, 8, 24\},\{ 2, 9, 25\},\{ 3, 10, 26\} \big\}$

4-colouring: {\tt 111122211112222333144323343434444}

\item
\begin{tabbing}
$\cP=\big\{$\=
$\{ 1, 2, 5\},\{ 17, 18, 21\},\{ 9, 10, 13\},\{ 28, 30, 16\},\{ 23, 25, 11\},\{ 27, 29, 15\}$,
\\ \> $\{ 31, 4, 22\},\{ 8, 14, 32\},\{ 0, 6, 24\},\{ 7, 12, 20\},\{ 19, 26, 3\}
\big\}$
\end{tabbing}

$\cS=\big\{ \{3, 8, 16\},\{ 2, 7, 15\},\{ 2, 9, 19\},\{ 3, 10, 20\},\{ 1, 12, 23\} \big\}$

4-colouring: {\tt 111122211112221332243343434344443}

\item
\begin{tabbing}
$\cP=\big\{$\=
$\{ 1, 2, 5\},\{ 8, 9, 12\},\{ 24, 25, 28\},\{ 31, 0, 19\},\{ 11, 13, 32\},\{ 18, 20, 6\}$,
\\ \> $\{ 16, 22, 7\},\{ 23, 29, 14\},\{ 30, 3, 21\},\{ 10, 17, 27\},\{ 4, 15, 26\}
\big\}$
\end{tabbing}

$\cS=\big\{ \{3, 8, 16\},\{ 2, 7, 15\},\{ 1, 6, 14\},\{ 3, 10, 20\},\{ 2, 9, 19\} \big\}$

4-colouring: {\tt 111122211112221332243343434344443}

\item
\begin{tabbing}
$\cP=\big\{$\=
$\{ 1, 2, 5\},\{ 20, 21, 24\},\{ 15, 16, 19\},\{ 7, 9, 28\},\{ 11, 13, 32\},\{ 6, 8, 27\}$,
\\ \> $\{ 31, 4, 22\},\{ 23, 29, 14\},\{ 12, 18, 3\},\{ 25, 30, 17\},\{ 26, 0, 10\}
\big\}$
\end{tabbing}

$\cS=\big\{ \{3, 8, 28\},\{ 2, 7, 27\},\{ 1, 8, 18\},\{ 3, 10, 20\},\{ 1, 12, 23\} \big\}$

4-colouring: {\tt 111122211112222333344431234334444}

\item
\begin{tabbing}
$\cP=\big\{$\=
$\{ 1, 2, 5\},\{ 8, 9, 12\},\{ 24, 25, 28\},\{ 31, 0, 19\},\{ 11, 13, 32\},\{ 18, 20, 6\}$,
\\ \> $\{ 16, 22, 7\},\{ 23, 29, 14\},\{ 30, 3, 21\},\{ 10, 17, 27\},\{ 4, 15, 26\}
\big\}$
\end{tabbing}

$\cS=\big\{ \{3, 8, 28\},\{ 2, 7, 27\},\{ 1, 6, 26\},\{ 3, 10, 20\},\{ 2, 9, 19\} \big\}$

4-colouring: {\tt 111122211112222333344431234334444}

\item
\begin{tabbing}
$\cP=\big\{$\=
$\{ 1, 3, 16\},\{ 8, 10, 23\},\{ 2, 0, 15\},\{ 4, 25, 28\},\{ 14, 17, 26\},\{ 6, 27, 30\}$,
\\ \> $\{ 7, 13, 21\},\{ 5, 11, 19\},\{ 18, 24, 32\},\{ 12, 22, 29\},\{ 9, 20, 31\}
\big\}$
\end{tabbing}

$\cS=\big\{ \{3, 10, 26\},\{ 2, 9, 25\},\{ 1, 2, 6\},\{ 2, 3, 7\},\{ 3, 4, 8\} \big\}$

4-colouring: {\tt 111112222111322223314343334344444}

\item
\begin{tabbing}
$\cP=\big\{$\=
$\{ 1, 3, 16\},\{ 2, 4, 17\},\{ 9, 11, 24\},\{ 7, 10, 19\},\{ 8, 29, 32\},\{ 15, 18, 27\}$,
\\ \> $\{ 22, 28, 14\},\{ 20, 26, 12\},\{ 6, 25, 0\},\{ 5, 21, 31\},\{ 13, 23, 30\}
\big\}$
\end{tabbing}

$\cS=\big\{ \{3, 10, 26\},\{ 1, 2, 6\},\{ 2, 3, 7\},\{ 3, 4, 8\},\{ 1, 12, 23\} \big\}$

4-colouring: {\tt 111112222111322223314343334344444}

\item
\begin{tabbing}
$\cP=\big\{$\=
$\{ 1, 3, 16\},\{ 2, 4, 17\},\{ 27, 29, 9\},\{ 7, 10, 19\},\{ 8, 11, 20\},\{ 12, 15, 24\}$,
\\ \> $\{ 22, 28, 14\},\{ 26, 32, 18\},\{ 31, 5, 21\},\{ 23, 30, 13\},\{ 0, 6, 25\}
\big\}$
\end{tabbing}

$\cS=\big\{ \{1, 2, 6\},\{ 2, 3, 7\},\{ 3, 4, 8\},\{ 1, 12, 23\},\{ 3, 10, 26\} \big\}$

4-colouring: {\tt 111112222111322223314343334344444}

\item
\begin{tabbing}
$\cP=\big\{$\=
$\{ 1, 3, 16\},\{ 8, 26, 28\},\{ 9, 11, 24\},\{ 19, 22, 31\},\{ 2, 5, 14\},\{ 6, 27, 30\}$,
\\ \> $\{ 7, 13, 32\},\{ 15, 23, 29\},\{ 4, 12, 18\},\{ 10, 17, 0\},\{ 20, 21, 25\}
\big\}$
\end{tabbing}

$\cS=\big\{ \{3, 10, 26\},\{ 2, 9, 25\},\{ 1, 2, 6\},\{ 3, 4, 8\},\{ 1, 12, 23\} \big\}$

4-colouring: {\tt 111112222111322223314343334344444}

\item
\begin{tabbing}
$\cP=\big\{$\=
$\{ 1, 3, 16\},\{ 11, 13, 26\},\{ 15, 17, 30\},\{ 19, 22, 10\},\{ 32, 2, 23\},\{ 21, 24, 12\}$,
\\ \> $\{ 25, 31, 6\},\{ 14, 20, 28\},\{ 27, 0, 8\},\{ 4, 5, 9\},\{ 7, 18, 29\}
\big\}$
\end{tabbing}

$\cS=\big\{ \{3, 10, 20\},\{ 2, 9, 19\},\{ 1, 8, 18\},\{ 3, 4, 8\},\{ 2, 3, 7\} \big\}$

4-colouring: {\tt 111112222111322223333444432334444}

\item
\begin{tabbing}
$\cP=\big\{$\=
$\{ 1, 3, 16\},\{ 17, 19, 32\},\{ 9, 11, 24\},\{ 22, 25, 13\},\{ 5, 8, 29\},\{ 27, 30, 18\}$,
\\ \> $\{ 31, 4, 23\},\{ 14, 20, 6\},\{ 15, 21, 7\},\{ 28, 2, 12\},\{ 26, 0, 10\}
\big\}$
\end{tabbing}

$\cS=\big\{ \{3, 10, 20\},\{ 1, 2, 6\},\{ 2, 3, 7\},\{ 3, 4, 8\},\{ 1, 12, 23\} \big\}$

4-colouring: {\tt 111112222111322223133434434334444}

\item
\begin{tabbing}
$\cP=\big\{$\=
$\{ 1, 3, 16\},\{ 23, 25, 5\},\{ 0, 2, 15\},\{ 28, 31, 19\},\{ 8, 11, 32\},\{ 6, 9, 30\}$,
\\ \> $\{ 4, 10, 29\},\{ 20, 26, 12\},\{ 21, 27, 13\},\{ 7, 14, 24\},\{ 17, 18, 22\}
\big\}$
\end{tabbing}

$\cS=\big\{ \{3, 10, 20\},\{ 2, 9, 19\},\{ 1, 2, 6\},\{ 3, 4, 8\},\{ 1, 12, 23\} \big\}$

4-colouring: {\tt 111112222111322223133434434334444}

\item
\begin{tabbing}
$\cP=\big\{$\=
$\{ 1, 3, 16\},\{ 5, 7, 20\},\{ 15, 17, 30\},\{ 28, 31, 19\},\{ 11, 14, 2\},\{ 21, 24, 12\}$,
\\ \> $\{ 4, 10, 29\},\{ 26, 32, 18\},\{ 0, 6, 25\},\{ 22, 23, 27\},\{ 8, 9, 13\}
\big\}$
\end{tabbing}

$\cS=\big\{ \{3, 10, 20\},\{ 2, 9, 19\},\{ 1, 8, 18\},\{ 3, 4, 8\},\{ 1, 12, 23\} \big\}$

4-colouring: {\tt 111112222111322223133434434334444}

\item
\begin{tabbing}
$\cP=\big\{$\=
$\{ 1, 3, 16\},\{ 20, 22, 2\},\{ 30, 32, 12\},\{ 7, 10, 31\},\{ 26, 29, 17\},\{ 15, 18, 6\}$,
\\ \> $\{ 19, 25, 11\},\{ 8, 14, 0\},\{ 21, 27, 13\},\{ 4, 5, 9\},\{ 23, 24, 28\}
\big\}$
\end{tabbing}

$\cS=\big\{ \{3, 10, 20\},\{ 2, 9, 19\},\{ 1, 8, 18\},\{ 3, 4, 8\},\{ 1, 12, 23\} \big\}$

4-colouring: {\tt 111112222111322223133434434334444}

\item
\begin{tabbing}
$\cP=\big\{$\=
$\{ 1, 3, 21\},\{ 26, 28, 13\},\{ 18, 20, 5\},\{ 7, 10, 19\},\{ 29, 32, 8\},\{ 12, 15, 24\}$,
\\ \> $\{ 16, 22, 30\},\{ 17, 23, 31\},\{ 0, 6, 14\},\{ 4, 11, 27\},\{ 2, 9, 25\}
\big\}$
\end{tabbing}

$\cS=\big\{ \{3, 10, 26\},\{ 1, 2, 6\},\{ 2, 3, 7\},\{ 3, 4, 8\},\{ 1, 12, 23\} \big\}$

4-colouring: {\tt 111112222111322213334342434434434}

\item
\begin{tabbing}
$\cP=\big\{$\=
$\{ 1, 13, 19\},\{ 23, 11, 5\},\{ 3, 15, 21\},\{ 2, 4, 7\},\{ 12, 14, 17\},\{ 28, 30, 0\}$,
\\ \> $\{ 6, 16, 25\},\{ 10, 20, 29\},\{ 8, 18, 27\},\{ 24, 31, 32\},\{ 9, 22, 26\}
\big\}$
\end{tabbing}

$\cS=\big\{ \{25, 32, 0\},\{ 26, 0, 1\},\{ 10, 23, 27\},\{ 11, 24, 28\},\{ 1, 12, 23\} \big\}$

4-colouring: {\tt 111112221212121232334433444343434}

\item
\begin{tabbing}
$\cP=\big\{$\=
$\{ 32, 0, 16\}$,
$\{ 1, 10, 15\}$,
$\{ 2, 22, 23\}$,
$\{ 3, 5, 9\}$,
$\{ 4, 12, 29\}$,
$\{ 6, 7, 18\}$,
\\ \>
$\{ 8, 27, 30\}$,
$\{ 11, 14, 24\}$,
$\{ 13, 20, 28\}$,
$\{ 17, 26, 31\}$,
$\{ 19, 21, 25\}
\big\}$
\end{tabbing}

4-colouring: {\tt 111111222211332221333433434444244}
\end{enumerate}

The following system is one of the $13$ $3$-chromatic 1-rotational KTS(33) over the cyclic group described in~\cite{BurattiZuanni}.  It is number 59a in the list of these systems~\cite{BurattiZuanniAppendix}.
We give base blocks over the point set $\mathbb{Z}_{32} \times \{\infty\}$; the remaining blocks can be obtained by developing modulo $32$.  The $3$-colouring is indicated by a sequence of length $33$, with the first 32 entries being the colours of $0, \ldots, 31$ in order and the last entry the colour of $\infty$.

\[
\begin{array}{l}
\cP=\big\{ \{1,3,9\}, \{2,5,27\}, \{6,15,20\}, \{7,24,28\}, \{13,14,26\}, \{\infty,0,16\} \big\} \\[1ex]
\mbox{3-colouring: } {\tt 111211312131322322112233123233132}
\end{array}
\]

\end{document}